\documentclass[11pt,a4paper,twoside]{article}
\usepackage{fullpage}
\usepackage[utf8]{inputenc}
\usepackage[T1]{fontenc}
\usepackage{float}
\usepackage{lmodern} 
\usepackage[english]{babel}
\usepackage{amsmath,amsthm,amssymb,amsfonts,mathtools,mathrsfs}
\usepackage{breqn}
\usepackage{tabularx}
\usepackage{graphicx}
\usepackage{multirow}
\usepackage{url}
\usepackage{hyperref}
\graphicspath{
{Images/}
{img/}
{Fig/}
}
\usepackage{epstopdf}
\usepackage[font=small,labelfont=bf,tableposition=top]{caption}
\DeclareCaptionLabelFormat{andtable}{#1~#2  \&  \tablename~\thetable}
\usepackage{subcaption}
\usepackage{array}
\usepackage{xcolor}
\usepackage[subnum]{cases}
\usepackage{algorithmic}
\ifpdf
  \DeclareGraphicsExtensions{.eps,.pdf,.png,.jpg}
\else
  \DeclareGraphicsExtensions{.eps}
\fi
\newtheorem{theorem}{Theorem}[section]

\newtheorem{proposition}[theorem]{Proposition}

\newcommand{\cA}{\ensuremath{\mathcal{A}}}
\newcommand{\cB}{\ensuremath{\mathcal{B}}}
\newcommand{\cC}{\ensuremath{\mathcal{C}}}
\newcommand{\cD}{\ensuremath{\mathcal{D}}}
\newcommand{\cE}{\ensuremath{\mathcal{E}}}
\newcommand{\cF}{\ensuremath{\mathcal{F}}}
\newcommand{\cG}{\ensuremath{\mathcal{G}}}

\newcommand{\cS}{\ensuremath{\mathcal{S}}}

\newcommand{\bI}{\ensuremath{\mathbb{I}}}

\newcommand{\bP}{\ensuremath{\mathbb{P}}}

\newcommand{\bR}{\ensuremath{\mathbb{R}}}

\newcommand{\bU}{\ensuremath{\mathbb{U}}}

\newcommand{\DDM}{ \ensuremath{\texttt{DDM}}} 
\newcommand{\SDC}{ \ensuremath{\texttt{SDC}}} 

\newcommand{\NI}{\ensuremath{\underline{N}}}   
\newcommand{\nfi}{\ensuremath{\underline{n}}}  
\newcommand{\nFine}{\ensuremath{\underline{\underline n}}} 

\newcommand{\seq}{\ensuremath{\text{seq} }}

\newcommand{\eff}{\ensuremath{\textbf{eff} }}
\newcommand{\speedup}{\ensuremath{\textbf{speed-up} }}
\newcommand{\AP}{\ensuremath{\text{AP} }}
\newcommand{\CP}{\ensuremath{\text{CP} }}

\newcommand{\cost}{\ensuremath{\textbf{cost} }}

\newcommand{\eps}{\ensuremath{\varepsilon}}

\newcommand{\vepsdg}{\ensuremath{\varepsilon_{\cG}^{\phantom{\ast}} }}
\newcommand{\vepsdgs}{\ensuremath{\varepsilon_{\cG} }}

\newcommand{\om}[1]{\textcolor{black}{#1}}
\newcommand{\rev}[1]{\textcolor{black}{#1}}
\newcommand{\remember}[1]{\textcolor{black}{#1}}

\newcommand{\ym}[1]{\textcolor{black}{#1}}

\newcommand{\corr}[1]{\textcolor{black}{#1}}

\newcommand{\TheTitle}{An Adaptive Parareal Algorithm}

\title{\TheTitle}

\author{Y.~Maday, O.~Mula}
\date{}

\begin{document}

\maketitle

\footnotetext[1]{\corr{This work was funded by the CINE-PARA project ANR-15-CE23-0019.}}

\begin{abstract}
In this paper, we consider the problem of accelerating the numerical simulation of time dependent problems by time domain decomposition. The available algorithms enabling such decompositions present severe efficiency limitations and are an obstacle for the solution of large scale and high dimensional problems. Our main contribution is the  improvement of the parallel efficiency of the parareal in time method. The parareal method is based on combining predictions made by a numerically inexpensive solver (with coarse physics and/or coarse resolution) with corrections coming from an expensive solver (with high-fidelity physics and high resolution). At convergence, the algorithm provides a solution that has the fine solver's high-fidelity physics and high resolution. In the classical version, the fine solver has a fixed high accuracy which is the major obstacle to achieve a competitive parallel efficiency. In this paper, we develop an adaptive variant that overcomes this obstacle \corr{by dynamically increasing the accuracy of the fine solver across the parareal iterations.} We theoretically show that the parallel efficiency becomes very competitive \rev{in the ideal case where the cost of the coarse solver is small, thus proving that the only remaining factors impeding full scalability become the cost of the coarse solver and communication time.} \corr{The developed theory has also the merit of setting a general framework to understand the success of several extensions of parareal based on iteratively improving the quality of the fine solver and re-using information from previous parareal steps.} \rev{We illustrate the actual performance of the method in stiff ODEs, which are a challenging family of problems since the only mechanism for adaptivity is time and efficiency is affected by the cost of the coarse solver.}
\end{abstract}

%

\section{Introduction}
\label{sec:intro}
Solving complex models with high accuracy and within a reasonable computing time has motivated the search for numerical schemes that exploit efficiently parallel computing architectures.
\ym{In this paper, the model consists of a }  Partial Differential Equation (PDE) \ym{set on a domain $\cD$. In this context}, one of the main ideas to parallelize a simulation is to break the problem into subproblems defined over subdomains of a partition of \ym{$\cD$}.
The domain can potentially have high dimensionality and be composed of different variables like space, time, velocity or even more specific variables for some problems.
While there exist algorithms with very good scalability properties for the decomposition of the spatial variable in elliptic and saddle-point problems (see \cite{quarteroni1996domain} or  \cite{TW2005} for an overview), the same cannot be said for the decomposition of time of even simple systems of ODEs. This is despite the fact that research on time domain decomposition is currently very active and has by now a history of at least 50 years (back to at least \cite{Nievergelt1964}) during which several algorithms have been explored (see \cite{Gander2015} for an overview). As a consequence, time domain decomposition is to date only a secondary option when it comes to deciding what algorithm/method distributes the tasks in a parallel cluster.

The main goal of this work is to address this efficiency limitation in the framework of one particular scheme: the parareal in time algorithm. The method was first introduced in  \cite{Maday2001} and has been well accepted by the community because it is easily applicable to a relatively large spectrum of problems. (Some specific difficulties are nevertheless encountered on certain types of PDEs as reported in, e.g., \cite{Dai,Farhat2003} for hyperbolic systems or \cite{bal2008symplectic, dai2013symmetric} for hamiltonian problems). Another ingredient for its success is that, even though its scalability properties are limited, they are in general competitive in comparison with other methods. Without entering into very specific details of the algorithm at this stage, we can summarize the procedure by saying that we build iteratively a sequence to approximate the exact solution of the problem by a predictor-corrector algorithm. \om{At every iteration, predictions are made by a solver which has to be as numerically inexpensive as possible \ym{since it is run on the full time interval}. It usually involves coarse physics and/or coarse resolution. Corrections involve an expensive solver with high-fidelity physics and high resolution which is propagated in parallel over small time subdomains. In the classical version of parareal, the fine solver has a fixed high accuracy across all iterations. It is set to the one that we would use to solve the dynamics at the desired accuracy with a purely sequential solver. It is well-known that this point is the major obstacle to achieve better parallel efficiency. In this paper, we propose an adaptive variant where the accuracy of the fine solver is increased across the iterations. \remember{Our main goal is to show that this new point of view overcomes the obstacle of the cost of the fine solver and that the only remaining factors limiting high performance become the cost of the coarse solver and communication time.}} \ym{We refer to, e.g., \cite{carlberg2019data} for contributions on the lowering of the cost of that coarse solver.}

\om{
We present in section \ref{sec:algo} the new adaptive point of view. This requires to formulate an idealized version of the parareal algorithm in an infinite dimensional function space where the fine propagations are replaced by the exact ones (section \ref{sec:ideal}). Since this scheme is obviously not implementable in practice, we formulate a feasible ``perturbed'' version that involves approximations of the exact propagations at increasing accuracy across the iterations (section \ref{sec:feasible}). The accuracies are tightened in such a way that the feasible adaptive algorithm converges at the same rate as the ideal one and with a near-minimal numerical cost. The identified tolerances involve quantities that are difficult to estimate in practice. In addition, they may not be optimal because they are derived from a theoretical convergence analysis based on abstract conditions for the coarse and fine solvers. We bridge this gap between theory and actual implementation by proposing practical guidelines to set these tolerances. We next explain in section \ref{sec:adaptivity} how the new formulation invites to use adaptive schemes not only in the time variable, but also in other variables that may be involved in the dynamics. The performance of the algorithm could also be enhanced by re-using informations from previous iterations in order to limit the cost of internal solvers. The techniques for this will strongly depend on the nature of the specific problem. We discuss common situations in Appendix \ref{sec:enrich-input}. We close section \ref{sec:connection-classical} by listing the main advantages of the new framework and how the classical parareal \ym{paradigm} can be formulated with the optics of the new standpoint.
}

\om{The parallel performance of the adaptive scheme is difficult to predict a priori but in section \ref{sec:speedup} we carry a discussion where we show that it will always be superior to the classical approach. \remember{In the idealized situation where the cost of the coarse solver and \rev{communication delays are negligible}, we show that the algorithm would exhibit a very high parallel efficiency.}
}

\rev{
We emphasize that our theory is general in the sense that it is applicable to ODEs and also to PDEs involving time, space and possibly other variables. We defer to a future work the presentation of a numerical PDE example since it requires the deployment of space-time adaptive methods which is a challenging topic in itself and the techniques usually depend very specifically on the problem nature. Instead, we illustrate the performance of the algorithm on stiff ODEs in section \ref{sec:numerics}. They are a challenging family of problems because the only source of adaptivity is time and they do not allow to use a very inexpensive coarse solver. The tested ODEs are the Brusselator, the Van der Pol, the Oregonator equations, and an SEIR model which has very recently been proposed in \cite{covid-model-2020} to model the spread of the COVID-19 virus in the Wuhan city area. The code to reproduce our results (and experiment with other ODEs) is available online\footnote{Link to the code: \href{https://plmlab.math.cnrs.fr/mulahernandez/parareal-adaptive}{https://plmlab.math.cnrs.fr/mulahernandez/parareal-adaptive}}.
}
\rev{Our first two examples are relatively stiff, while the other two are highly stiff, and serve to illustrate the limitations of the approach. We show that in the relatively stiff problems, the adaptive parareal algorithm performs between 2 to 3 times better than the classical one when the solution is approximated at high accuracy. In addition, we confirm in these two examples that the only remaining obstacles to achieve a very competitive performance are the cost of the coarse solver and communication time between processors. Due to the nature of the algorithm, it is not clear how to overcome these limitations, especially the one coming from the coarse solver. This issue will come at the forefront for future research since, as our two highly stiff examples illustrate, the cost of the coarse solver may even prevent from obtaining any speed-up at all. We show that if we could find an inexpensive coarse solver, perhaps based on empirical data or on good back-of the enveloppe calculations, our adaptive algorithm would yield interesting speeds-ups even in highly stiff cases.}

In addition to the improvement in parallel efficiency \rev{(except, of course, in the above discussed extreme cases)}, the adaptive version of parareal brings an important conceptual novelty to the field of time domain decomposition which is the one of \emph{error controlled computations}. By this we mean that the deviation of the numerical result from the exact continuous solution is certifiably quantified and set to meet a given target accuracy with respect to a \emph{problem relevant norm}. This requires the formulation of the algorithm at an infinite dimensional level as is done in this paper. This point of view is fundamentally different from the fully discrete setting in which the parareal algorithm has always been thought of in practice. That is, we first fix a discretization (often on a uniform grid) and then speed-up the computation of the discrete evolution with parareal. In this way, we do not have any rigorous control of the error with respect to the actual continuous solution and we do not have any systematic procedure to minimize the number of degrees of freedom.

We conclude this introduction by some bibliographical remarks. To the best of our knowledge, the current abstract and broad formulation of an adaptive version of parareal has never been proposed in the literature. However, previous works have instantiated in a variety of particular applications the idea of re-using information from previous parareal iterations, either with the purpose of improving the quality of the fine solver or to build an initial guess of internal iterative routines. Among the most relevant ones stand the coupling of the parareal algorithm with spatial domain decomposition (see \cite{madayTuriniciDDM,guetatPhD,ABGM2017}), the combination of the parareal algorithm with iterative high order methods in time like spectral deferred corrections (see \cite{MW2008,Minion2010,MSBER2015}). \corr{Parareal has also been combined with multigrid iterative techniques which, in addition, involve a hierarchy of space-time meshes (see \cite{FFKMS2014,FMOS2017}).} In a similar spirit, we can also cite applications of the parareal algorithm to solve optimal control problems where information from previous steps is used (see \cite{madayTuriniciDDM,MST2007}). In appendix \ref{sec:enrich-input}, we briefly explain in what sense the \corr{above  strategies} can be seen as particular instances of the current approach and how our viewpoint could help to give them more solid theoretical foundations. \corr{The idea of re-using information as a starting guess for internal iterative solvers is also discussed in the Appendix. It has been explored in several different works, e.g.~\cite{FFKMS2014, MSBER2015}, and \cite{mulaPhD} provides a convergence analysis in simple situations (a more complete analysis will be proposed in a forthcoming work). }
\section{An adaptive parareal algorithm}
\label{sec:algo}
In this section, after introducing some preliminary notations in section \ref{sec:preliminaries}, we formulate an ideal parareal scheme on an infinite dimensional functional setting (section \ref{sec:ideal}). We then present feasible realizations involving a fine solver whose accuracy is adaptively increased across the iterations (section \ref{sec:feasible}). We prove that the feasible adaptive algorithm converges at the same rate as the ideal one provided that the tolerances of the fine solver are increased at certain rate which will be discussed. Finally, we discuss how the new paradigm can be realized thanks to adaptive schemes and\ym{/or} the re-use of information from previous steps (section \ref{sec:adaptivity} and appendix \ref{sec:enrich-input}). In section \ref{sec:connection-classical}, we connect the new adaptive formulation with the classical parareal algorithm and list the main advantages of the new standpoint.

\subsection{Setting and preliminary notations}
\label{sec:preliminaries} Let $\bU$ be a Banach space of functions defined over a domain 
$\Omega\subset\bR^d$ ($d\geq1$), e.g. $\bU = L^2(\Omega)$.
Let
\begin{equation*}
\cE:[0,T]\times[0,T]\times \bU\to \bU
\end{equation*}
be a propagator, that is, an operator such that, for any given time $t \in [0, T]$, $s \in [0,T-t]$ and any function $w\in \bU$, $\cE(t,s,w)$ takes $w$ as an initial value at time $t$ and propagates it at time $t+s$. We assume that $\cE$ satisfies the semi group property
\begin{equation*}
\cE(r,t-r,w)=\cE(s,t-s, \cE(r, s-r,w)  ),\quad \forall w\in \bU, \forall (r, s, t) \in[0,T]^3,\ r<s<t.
\end{equation*}
We further assume that $\cE$ is implicitly defined through the solution $u\in \cC^1([0,T],\bU)$ of the time-dependent problem
\begin{align}
\label{eq:pde}
u'(t) + \mathcal{A}\left(t,u(t)\right) = 0,\quad t\in [0, T],
\end{align}
where $\cA$ is an operator \ym{from} $[0,T]\times \bU$ \ym{into $\bU$ with adequate regularity we shall detail latter}. Then, given any $w\in \bU$, $\cE(t,s,w)$ denotes the solution to \eqref{eq:pde} at time $t+s$ with initial condition $w$ at time $t\ge 0$. In our problem of interest, we study the evolution given by \eqref{eq:pde} when the initial condition is $u(0)\in\bU$. Note that $\cE$ could also be associated to a discretized version of the evolution equation or be defined through an operator that is not necessary related to an evolution equation (see \cite{GG2018}).

Since, in general, the problem does not have an explicit solution, we seek to approximate it at a given  target accuracy. For any initial value $w\in \bU$, any $t\in[0,T[$, $s\in[0,T-t]$ and any $\zeta>0$ we denote by $[\cE(t,s,w);\zeta]$ an element of $\bU$ that approximates $\cE(t,s,w)$ such that we have
\begin{equation}
\label{eq:approxE}
\Vert \cE(t,s,w) - [\cE(t,s,w);\zeta] \Vert \leq \zeta\, s\, (1+\Vert w \Vert),
\end{equation}
where, here and in the following, $\Vert\cdot\Vert$ denotes the norm in $\bU$.
Any realization of $[\cE(t,s,w);\zeta]$ involves three main ingredients:
\begin{itemize}
\item[i)] a numerical scheme to discretize the time dependent problem \eqref{eq:pde}  (e.g.~an Euler scheme in time),
\item[ii)] a certain expected error size associated with the choice of the discretization (e.g.~error associated with the time step size of the Euler scheme),
\item[iii)] a numerical implementation to solve the resulting discrete systems (e.g.~conjugate gradient, Newton method, SSOR, \dots).
\end{itemize}
In the following, we will use the term \emph{solver} to denote a particular choice for i), ii) and iii). Given a solver $\cS$, we will use the same notation as for the exact propagator $\cE$ to express that $\cS(t,s,w)$ is an approximation of $\cE(t,s,w)$ with a certain accuracy $\zeta$. In other words, we can write
$ \cS(t,s,w) = 
[\cE(t,s,w);\zeta] $.

\subsection{An idealized version of the parareal algorithm}
\label{sec:ideal}
\corr{We introduce} a decomposition of the time interval $[0,T]$ into ${\NI}$  subintervals $[T_N, T_{N+1}]$, $N=0,\dots, {\NI}-1$. Without loss of generality, we will take them of uniform size $\Delta T= T /\NI$ which means that  $T_N = N \Delta T$ for $N=0,\dots, {\NI}$. For a given target accuracy $\eta>0$, the primary goal of the parareal in time algorithm is to build an approximation $\tilde u(T_N)$ of $u(T_N)$ such that
\begin{equation}
\label{eq:primaryGoal}
\max_{\ym{1}\leq N\leq \NI} \Vert u(T_N)-\tilde u(T_N)\Vert \leq \eta.
\end{equation}
 The classical way to achieve this is to set 
$$\tilde u(T_N)=\cS_{\seq}(0,T_N,u(0))=[\cE(0,T_N,u(0));\zeta],\quad \ym{1}\leq N \leq \NI,$$
where $\cS_{\seq}$ is some sequential solver in $[0,T]$ with $\zeta = \eta/(T (1+ \Vert u(0)\Vert))$ in (\ref{eq:approxE}).
Since this comes at the cost of solving over the whole time interval $[0,T]$,
the main goal of the parareal in time algorithm is to speed up the computing time, while maintaining the same target accuracy $\eta$. This is made possible by first decomposing the computations over the time domain. Instead of solving over $[0,T]$, we perform $\NI$ parallel solves over \ym{each} interval \ym{$(T_N, T_{N+1}]$}of size $\Delta T$. We next introduce an idealized version of it which will not be feasible in practice but will be the starting point of subsequent implementable versions. The algorithm relies on the use of a solver $\cG$ (known as the coarse solver) with the following properties involving the operator
\begin{equation*}
\delta \cG\coloneqq \cE-\cG.
\end{equation*}
\noindent
\textbf{Hypotheses (H)}: There exists constants $\vepsdg,\ C_c,\ C_d>0$ such that for any function $x,\ y \in \bU$ and for any $t\in[0,T[$ and $s \in[0,T-t]$,
\begin{subequations}\label{eq:hypH}
\begin{align}
&
\cG(t,s,x)= [\cE(t,s,x),\vepsdg] 
\quad\Leftrightarrow\quad
\|\delta\cG(t,s,x)\|\le s  (1+\|x\|) \vepsdg \label{eq:hyp1g}\\
&
\|\cG(t,s,x) - \cG(t,s,y)  \|\le (1+ C_c s) 
\|x-y\|,\label{eq:hyp3}\\
&
\|\delta\cG(t,s,x)-\delta\cG(t,s,y) \|\le C_d
s \vepsdg  \|x-y\| \label{eq:hyp4}
\end{align}
\end{subequations}
Note that these hypothesis are the classical abstract formulations of the properties of numerical schemes related to stability and accuracy. Hypothesis \eqref{eq:hyp3} is a Lipschitz condition and the quantity $\vepsdg$ is a small constant which, in the case of a Euler scheme, would be equal to the time step size.

The idealized version of the algorithm consists in building iteratively a series $(y^N_k)_k$ of approximations of $u(T_N)$ for $0\leq N\leq\NI$ following the recursive formula
\begin{equation}
\label{eq:paraPPexact}
\begin{cases}
\begin{aligned}
y_0^{N+1}
&=  \cG(T_N,\Delta T,y^N_0),\ &0\leq N \leq \underline{N}-1 \\
y^{N+1}_{k+1} 
&=  \cG(T_N,\Delta T,y^{N}_{k+1})  +\cE(T_N,\Delta T,y^{N}_{k})\\
&\qquad - \cG(T_N,\Delta T,y^{N}_{k}),\ &0\leq N \leq \underline{N}-1, \quad k\geq 0, \\
y_0^0
&= u(0).
\end{aligned}
\end{cases}
\end{equation}
At this point, several comments are in order. The first one is that the computation of $y^N_k$ only requires propagations with $\cE$ over intervals of size $\Delta T$. As follows from \eqref{eq:paraPPexact}, for a given iteration $k$, $\NI$ propagations of this size are required, each of them over distinct intervals $[T_N,T_{N+1}]$ of size $\Delta T$, each of them with independent initial conditions. Since they are independent from each other, they can be computed over $\NI$ parallel processors and the original computation over $[0,T]$ is decomposed into \ym{parallel computations over } $\NI$ subintervals of size $\Delta T$. The second observation is that the algorithm may not be implementable in practice because it involves the exact propagator $\cE$. Feasible instantiations consist of replacing $\cE(T_N,\Delta T,y^{N}_{k})$ by some approximation $[\cE(T_N,\Delta T,y^{N}_{k}),\zeta_k^N]$ with a certain accuracy $\zeta_k^N$ which has to be carefully chosen. We will come to this point in the next section. The third observation is to note that, in the current version of the algorithm, for all $N=0,\dots,\NI$, the exact solution $u(T_N)$   is obtained after exactly $k=N$ parareal iterations. This number can be reduced when we only look for an approximate solution with accuracy $\eta$. \om{Depending on the problem, the final number of iterations $K(\eta)$ can actually be much smaller than $\NI$. The convergence result of theorem \ref{th:convPPexact} and its proof are helpful to understand the main mechanisms driving the convergence of the algorithm and explaining its behavior.} To present it, we introduce the shorthand notation for the error norm
\begin{equation*}
E^N_k \coloneqq \Vert u(T_N)-y^N_k \Vert,\quad k\geq 0,\ 0\leq N \leq \NI,
\end{equation*}
and the quantities
\begin{equation*}
\mu =  \frac{{e^{C_c T}}}{C_d} \max_{0\leq N\leq \NI} (1+\Vert u(T_N)\Vert)  , 
\quad\text{and}\quad
\tau\coloneqq C_d  T  e^{ - C_c  \Delta T} \varepsilon_{\cG}.
\end{equation*}

\begin{theorem}
\label{th:convPPexact}
If $\cG$ and $\delta \cG$ satisfy Hypothesis \eqref{eq:hypH}, then,
\begin{equation}
\label{eq:convPlainBound}
\max_{0\leq N\leq \NI} 
\Vert u(T_N)-y^N_k \Vert
\leq \mu   \frac{\tau^{k+1}}{(k+1)!} ,\, \quad\forall k\geq 0.
\end{equation}
\end{theorem}
\begin{proof}
The proof is in the spirit of existing results from the literature (see \cite{Maday2001,Bal2003,MRS, GH2008}) but it is instructive to give it for subsequent developments in the paper. We introduce the following quantities
\begin{equation}
\label{eq:abg}
\begin{cases}
\alpha&\coloneqq C_d \vepsdg \Delta T \\
\beta&\coloneqq 1+C_c \Delta T \\
 \gamma&\coloneqq \Delta T \vepsdg  \max_{0\leq N\leq \NI} (1+\Vert u(T_N)\Vert)
 \end{cases}
\end{equation}
as shorthand notations for the proof.

If $k=0$, using definition \eqref{eq:paraPPexact} for $y^N_0$, we have for $0\leq N \leq \NI-1$, 
\begin{align*}
E^{N+1}_0
&= \Vert  y^{N+1}_0 - u(T_{N+1}) \Vert  \\
&= \| \cG(T_N,\Delta T,y^N_0) -  \cE(T_N,\Delta T,u(T_{N}))\| \\
&\leq \Vert \cG(T_N,\Delta T,y^N_0) - \cG(T_N,\Delta T,u(T_N))  \Vert
+ \Vert \cG(T_N,\Delta T,u(T_N)) - \cE(T_N,\Delta T,u(T_{N}))  \Vert \\
&\leq (1+C_c\Delta T) E^{N}_0 + \Delta T  \vepsdg (1+\Vert u(T_N)\Vert)\\
&\leq \beta E^N_0 + \gamma,
\end{align*}
where we have used \eqref{eq:hyp1g} and \eqref{eq:hyp3} to derive the second to last inequality.

For $k\geq1$, starting from \eqref{eq:paraPPexact}, we have
 \begin{align*}
y_k^{N+1} - u(T_{N+1})
&= \cG(T_N,\Delta T,y^{N}_{k}) + \cE(T_N,\Delta T,y^{N}_{k-1}) - \cG(T_N,\Delta T,y^{N}_{k-1}) - \cE(T_N,\Delta T,u(T_N)) \\
&=\cG(T_N,\Delta T,y^N_k) -  \cG(T_N,\Delta T,u(T_N))+ \delta\cG (T_N,\Delta T,y^{N}_{k-1}) -\delta\cG(T_N,\Delta T,u(T_N)).
\label{paraRB::eq::paraErr} \nonumber
\end{align*}
Taking norms and using \eqref{eq:hyp3}, \eqref{eq:hyp4}, we derive
\begin{equation*}
E_{k}^{N+1} 
\le \beta E^{N}_{k} + \alpha  E^{N}_{k-1},
\end{equation*}
Following \cite{GH2008}, we consider the sequence $(e^N_k)_{N,k\geq0}$ defined recursively as follows. For $k=0$,
\begin{equation}
\label{3M2}
e_0^N =
\begin{cases}
0, &\quad\text{if } N=0\\
\beta e^{N-1}_0+\gamma, &\quad\text{if } N\geq 1
\end{cases}
\end{equation}
and for $k\geq 1$,
\begin{equation}
\label{3M1}
e_k^N =
\begin{cases}
0, &\quad\text{if } N=0\\
\alpha e^{N-1}_{k-1}+ \beta e^{N-1}_k, &\quad\text{if } N\geq 1.
\end{cases}
\end{equation}
Since $E^N_k \leq e^N_k$ for $k\geq 0$ and $N=0,\dots, \NI$, we analyze the behavior of $(e^{N}_{k})$ to derive a bound for $E^N_k$. For this, we consider the generating function
$$
\rho_k(\xi) = \sum_{N\ge 0} e^{N}_{k} \xi^N.
$$
From \eqref{3M2} and \eqref{3M1} we get
\begin{equation*}
\begin{cases}
\rho_k(\xi) = \alpha \xi \rho_{k-1}(\xi)  + \beta\xi \rho_k(\xi),\quad k\geq 1\\
\rho_0(\xi) = \gamma\frac{\xi}{1-\xi} + \beta\xi \rho_0(\xi),
\end{cases}
\end{equation*}
from which we derive 
\begin{equation*}
\rho_k(\xi) = \gamma\alpha^k\frac{\xi^{k+1}}{(1-\xi)}\frac{1}{(1-\beta\xi)^{k+1}},\quad k\geq 0.
\end{equation*}
Since, $\beta\ge 1$, we can bound the term $(1-\xi)$ in the denominator by $(1-\beta\xi)$. Next, using the binomial expansion
\begin{equation}
\frac{1}{(1-\beta\xi)^{k+2}} = \sum_{j\ge 0} {k+1+j\choose j} \beta^j \xi^j
\label{eq:binomial}
\end{equation}
and identifying the term in $\xi^{N}$ in the expansion, we derive the bound
\begin{equation*}
e^N_k \le \gamma \alpha^k \beta^{N-k-1} {N\choose k+1}.
\end{equation*}
Hence, using definition \eqref{eq:abg} for $\alpha,\ \beta$ and $\gamma$,
\begin{equation*}
E^N_k\leq e^N_k \le \frac{{(1+C_c\Delta T)^{N-k-1} \max_{0\leq N\leq \NI} (1+|| u(T_N)\Vert)}  }{C_d (k+1)!}   \bigl[C_d   \varepsilon_{\cG} e^{- C_c \Delta T} T_N \Bigr]^{k+1},
\end{equation*}
which ends the proof of the theorem.
\end{proof}
Note that \ym{at least one} step  is not sharp in the above proof: it is the step where $1-\xi$ is replaced by $1-\beta\xi$. 
Note also that $\tau$ is the quantity driving convergence and its speed.

\corr{Introducing the quantity
$$
{\bar \varepsilon}_{\cG} \coloneqq \frac{e^{  C_c \Delta T}}{C_d  T },
$$
we can write
$$
\tau = \frac{ {\varepsilon}_{\cG} }{{\bar \varepsilon}_{\cG}}
$$
and we note that a sufficient condition to converge is that
\begin{equation}
\label{eq:sufficientG}
\tau<1
\quad\Leftrightarrow\quad
\vepsdg <   {\bar \varepsilon}_{\cG} .
\end{equation}
In other words, ${\bar \varepsilon}_{\cG}$ is the minimal accuracy that the coarse solver has to satisfy in order to guarantee convergence of the ideal parareal algorithm. In the following, we will work under the assumption that $\vepsdg$ satisfies \eqref{eq:sufficientG}.
}

\corr{
As we will see in the next section, ${\bar \varepsilon}_{\cG}$ plays also a critical role in certain convergence properties of the perturbed algorithm so we finish this section by discussing the behavior of ${\bar \varepsilon}_{\cG}$ depending on several scenarios. First, $C_c$ and $C_d$ are Lipschitz constants (fixed by the properties of the evolution problem) so they could be potentially large numbers. As a result, ${\bar \varepsilon}_{\cG}$ could be a large number and condition \eqref{eq:sufficientG} would not be very stringent. The value of ${\bar \varepsilon}_{\cG}$ can be small for very long time simulations where $T$ becomes large or if $\Delta T$ becomes small compared to $C_c$ (that is, if the number $\NI$ of processors becomes large).
}

\subsection{Feasible realizations of the parareal algorithm}
\label{sec:feasible}
Feasible versions of algorithm \eqref{eq:paraPPexact} involve approximations of $\cE(T_N,\Delta T,y^{N}_{k})$ with a certain accuracy $\zeta_k^N$. This leads to consider algorithms of the form
\begin{equation}
\label{eq:paraAP}
\begin{cases}
\begin{aligned}
y_0^{N+1}
&=  \cG(T_N,\Delta T,y^N_0),\ &0\leq N \leq \underline{N}-1 \\
y^{N+1}_{k+1} 
&=  \cG(T_N,\Delta T,y^{N}_{k+1})+[\cE(T_N,\Delta T,y^{N}_{k});\zeta^N_k]  \\
&\qquad- \cG(T_N,\Delta T,y^{N}_{k}),\quad &0\leq N \leq \underline{N}-1,\ k\geq 0, \\
y_0^0
&= u(0).
\end{aligned}
\end{cases}
\end{equation}
Since no feasible version will converge at a better rate than \eqref{eq:convPlainBound}, we analyze here what is the minimal accuracy $\zeta_k^N$ that preserves it. A result in this direction is given in the following theorem. It requires to introduce the quantity
\begin{equation*}
\nu_p \coloneqq
\frac{\max_{0\leq N\leq \NI} (1+\Vert y^N_{p}\Vert )}{\max_{0\leq N\leq \NI} (1+\Vert u(T_N)\Vert) },\quad \forall p\geq0.
\end{equation*}
\ym{which tends to $1$ as $p\to\infty$.}
\begin{theorem}
\label{th:convAP}
Let $\cG$ and $\delta \cG$ satisfy Hypothesis \eqref{eq:hypH}. Let $k\geq0$ be any given positive integer. If for all $0\leq p <k$ and all $0\leq N < \NI$, the approximation $[\cE(T_N,\Delta T,\zeta^N_p)]$ has accuracy
\begin{equation}
\zeta^N_p\leq \zeta_p
\coloneqq  \frac{\eps_{\cG}^{p+2}}{(p+1)! \nu_p},
\label{eq:zeta}
\end{equation}
then the $(y^N_k)_{N}$ of the feasible parareal scheme \eqref{eq:paraAP} satisfy
\begin{equation}
 \label{eq:convFadaptBoundbis}
\max_{0\leq N\leq \NI}  
\Vert u(T_N)-y^N_k \Vert
\leq
 \mu
\frac{\tilde \tau^{k+1}}{(k+1)!},
\end{equation}
with
\begin{align*}
\tilde \tau\coloneqq \corr{ \tau + \eps_{\cG} }. 
\end{align*}
\end{theorem}
Let us make a couple of remarks before giving the proof of the theorem. \corr{First, the sufficient condition to converge is now
\begin{equation}
\tilde \tau<1
\quad\Leftrightarrow\quad
\vepsdg <  \frac{ {\bar \varepsilon}_{\cG}}{ 1 + {\bar \varepsilon}_{\cG} }
\end{equation}
so the minimal accuracy required for the coarse solver is stronger than in \eqref{eq:sufficientG} for the ideal case. Note however that when ${\bar \varepsilon}_{\cG}$ is small (roughly, ${\bar \varepsilon}_{\cG} \leq 1$), the condition on $\vepsdg$ is similar in the ideal and perturbed case.}

\corr{
Second, comparing \eqref{eq:convPlainBound} and \eqref{eq:convFadaptBoundbis}, the rate of convergence $\tilde \tau$ of the feasible parareal algorithm deviates from $\tau$, the ideal one, by a factor 
$$
\frac{\tilde \tau}{\tau} = \frac{\tau+\varepsilon_{\cG}}{\tau} = 1 + \frac{e^{C_c\Delta T}}{ C_d T} = 1 + {\bar \eps}_{\cG}.
$$
The parameter $ {\bar \eps}_{\cG}$ plays again a critical role in the convergence properties and determines whether convergence is close to the ideal rate $\tau$, or deviates from it by a potentially important factor.
}

%
\begin{proof}
The proof follows the same lines as the one for theorem \ref{th:convPPexact} and $E^N_k$, $\alpha,\ \beta,\ \gamma$ are defined exactly as before. In addition, it will be useful to introduce the sequence
\begin{equation*}
\begin{cases}
g_k &= \zeta_{k}\Delta T \max_{0\leq N\leq \NI}(1+\Vert y_{k}^N\Vert ),\quad \forall k\geq 0 \\
g_{-1}&= \gamma
\end{cases}
\end{equation*}
We concentrate on the case $k\geq 1$ since the case $k=0$ is identical as in theorem \ref{th:convPPexact}. For  $k\geq 1$, using \eqref{eq:paraAP}, we have
 \begin{align*}
y_k^{N+1} - u(T_{N+1})
&= \cG(T_N,\Delta T,y^{N}_{k}) -  \cG(T_N,\Delta T,u(T_N))- \cG(T_N,\Delta T,y^{N}_{k-1}) \nonumber\\
&\qquad+ \cG(T_N,\Delta T,u(T_N)) + [\cE(T_N,\Delta T,y^{N}_{k-1});\zeta^N_{k-1}]
-\cE(T_N,\Delta T,u(T_N)) \nonumber\\
&=\cG(T_N,\Delta T,y^N_k) -  \cG(T_N,\Delta T,u(T_N))+ \delta\cG (T_N,\Delta T,y^{N}_{k-1}) \nonumber\\
&\qquad -\delta\cG(T_N,\Delta T,u(T_N)) +[\cE(T_N,\Delta T,y^{N}_{k-1});\zeta^N_{k-1}] - \cE(T_N,\Delta T,y^{N}_{k-1}).
\label{paraRB::eq::paraErr} \nonumber
\end{align*}
Taking norms, using \eqref{eq:hyp3}, \eqref{eq:hyp4} and the definition \eqref{eq:approxE} applied to $[\cE(T_N,\Delta T,y^{N}_{k-1});\zeta^N_{k-1}]$, we derive
\begin{align*}
E_{k}^{N+1}   
&\leq [1+C_c \Delta T] E^{N}_{k} +  C_d\Delta T \vepsdg   E^{N}_{k-1} + \zeta^N_{k-1}\Delta T \max_{0\leq N\leq \NI}(1+\Vert y_{k-1}^N\Vert ) \\
&\leq \beta E^{N}_{k} + \alpha E^{N}_{k-1} + g_{k-1}.
\end{align*}
Similarly to theorem \ref{th:convPPexact}, we introduce
the sequence $(\tilde e^N_k)_{N,k\geq0}$ defined for $k=0$ as $\tilde e_0^N = e_0^N$ for all $N\geq 0$ and for $k\geq 1$,
\begin{equation*}
\tilde e_k^N =
\begin{cases}
0, &\quad\text{if } N=0\\
\alpha\tilde e^{N-1}_{k-1} +\beta\tilde e^{N-1}_{k}+g_{k-1}, &\quad\text{if } N\geq 1
\end{cases}
\end{equation*}
The associated generating function $\tilde\rho_k$ satisfies
\begin{equation*}
\begin{cases}
\tilde\rho_k(\xi) &= \alpha \xi\tilde\rho_{k-1} (\xi) + \beta \xi \tilde\rho_k(\xi) + g_{k-1} \frac{\xi}{1-\xi}, \quad \forall k\geq 1, \\
\tilde\rho_{0} (\xi) &= \rho_{0} (\xi) = \frac{\gamma\xi}{(1-\xi)(1-\beta\xi)}.
\end{cases}
\end{equation*}
Hence 
\begin{align*}
\tilde\rho_k(\xi)
&= \left(\frac{\alpha \xi}{1-\beta\xi}\right)\tilde\rho_{k-1} (\xi)  + \frac{\xi}{(1-\xi)(1-\beta\xi)} g_{k-1} \\
&= \left(\frac{\alpha \xi}{1-\beta\xi}\right)^k \tilde\rho_{0} (\xi)  + \frac{\xi}{(1-\xi)(1-\beta\xi)} \sum_{\ell=0}^{k-1}   \left(\frac{\alpha \xi}{1-\beta\xi}\right)^\ell   g_{k-1-\ell}
\end{align*}
By replacing again at the denominator the factor $(1-\xi)$ by $(1-\beta\xi)$ and using the binomial expansion \eqref{eq:binomial}, we derive the bound
\begin{align*}
\tilde \rho_k(\xi) 
&\leq \gamma\alpha^k\xi^{k+1} \sum_{j\ge0} {k+1+j\choose j} \beta^j\xi^j + \sum_{\ell=0}^{k-1} \alpha^\ell\xi^{\ell+1}g_{k-1-\ell}\sum_{j\ge0}{\ell+1+j\choose j} \beta^j\xi^j, \\
 &=  \sum_{j\ge0} \sum_{\ell=0}^{k}  \alpha^\ell g_{k-1-\ell}  \beta^j {\ell+1+j\choose j}  \xi^{\ell+1+j}
\end{align*}
where we have used that $g_{-1}= \gamma$.
The coefficient associated to the term $\xi^N$ above gives the inequality
\begin{equation*}
\tilde e^N_k \leq \sum_{\ell=0}^{k} \alpha^\ell g_{k-1-\ell}\beta^{N-\ell-1}{N\choose \ell+1},\quad \forall k\geq 1,
\end{equation*}
From the definition of $\zeta_\ell$, we have that
$$
g_\ell \leq \frac{\Delta T \varepsilon_{\cG}^{\ell+2}}{(\ell+1) !}\max_{0\leq N\leq \NI} (1+\Vert u(T_N)\Vert).
$$
Therefore, recalling the definition \eqref{eq:abg} of $\alpha,\ \beta$ and $\gamma$, we derive
\begin{align*}
\tilde e^N_k
&\leq \frac{\varepsilon_{\cG}^{k+1}\max_{0\leq N\leq \NI} (1+\Vert u(T_N)\Vert) }{C_d} \sum_{\ell=0}^k (C_d \Delta T)^{\ell+1} \frac{(1+C_c \Delta T)^{N-\ell-1}}{(k-\ell)!}  {N\choose \ell+1} \\
&\leq \frac{\varepsilon_{\cG}^{k+1}\max_{0\leq N\leq \NI} (1+\Vert u(T_N)\Vert) }{C_d}  \sum_{\ell=0}^k \left(C_d T e^{-C_C\Delta T}\right)^{\ell+1} \frac{e^{C_C T}}{(\ell+1) ! (k-1-\ell) !} \\
&\leq   \frac{\max_{0\leq N\leq \NI} (1+\Vert u(T_N)\Vert) e^{C_cT}}{C_d (k+1)!}\left( (1+C_d T e^{-C_c\Delta T}) \varepsilon_{\cG}\right)^{k+1}\corr{=   \mu \frac{\tilde \tau ^{k+1}}{(k+1)!},}
\end{align*}
\corr{where we have used the definition of $\mu$ and $\tilde \tau$ in the last line.} This inequality ends the proof since $E^N_k\leq \tilde e^N_k$ for $N=0,\dots,\NI$.
\end{proof}
\subsection{Practical realization of $[\cE(T_N,\Delta T,y^N_k),\zeta^N_k]$}
\label{sec:adaptivity}
Since the accuracy $\zeta^N_k$ \ym{needs to}  improve with $k$, the most natural way to build the approximations $[\cE(T_N,\Delta T,y^N_k),\zeta^N_k]$ is with adaptive techniques and with adaptive refinements at every step $k$. The implementation ultimately rests on the use of a posteriori error estimators. It opens the door to local time step adaptation in the parareal algorithm as well as spatial coarsening or refinement if the problem involves additional spatial variables.

In principle, as $\zeta^N_k$ decreases with $k$, the numerical cost increases in terms of degrees of freedom and also in terms of computing time. \ym{This actually reveals the key  idea of this new approach which is that we would like that only the last fine solver is expensive and the cost of the previous ones is a small fraction of the cost of the last one (we refer to the next sub-section for a more precise statement)}. By re-using information from previous iterations, we can limit the cost of internal solvers required in $[\cE(T_N,\Delta T,y^N_k),\zeta^N_k]$ and enhance the speed-up. This depends of course on the nature of the specific problem. We discuss several common situations in Appendix \ref{sec:enrich-input}.

\subsection{Connection to the classical formulation of the parareal algorithm and advantages of the current view-point}
\label{sec:connection-classical}
In the original version of the algorithm, $\cE(T_N,\Delta T,y^{N}_{k})$ is approximated with an accuracy $\zeta^N_k = \zeta_{\cF}$ which is kept \emph{constant} in $N$ and across the parareal iterations $k$. This has usually been done by using a solver $\cF$ defined in the same spirit as $\cG$, but satisfying Hypothesis \eqref{eq:hypH}  with a better accuracy $\eps_\cF <\vepsdg$. We have in this case
\begin{equation*}
[\cE(T_N,\Delta T,y^{N}_{k});\zeta_{\cF}]=\cF(T_N,\Delta T,y^{N}_{k})
\end{equation*}
and we recover the classical algorithm (see \cite{baffico2002parallel} and \cite{Bal2002})
\begin{equation*}
\begin{cases}
\begin{aligned}
y_0^{N+1}
&=  \cG(T_N,\Delta T,y^N_0),\ &0\leq N \leq \underline{N}-1 \\
y^{N+1}_{k+1} 
&=  \cG(T_N,\Delta T,y^{N}_{k+1}) \\
&\qquad+\cF(T_N,\Delta T,y^{N}_{k}) - \cG(T_N,\Delta T,y^{N}_{k}),\ &0\leq N \leq \underline{N}-1, \quad k\geq 0, \\
y_0^0
&= u(0).
\end{aligned}
\end{cases}
\end{equation*}
Compared to this classical version of the parareal algorithm, the adaptive approach offers the following important advantages:
\begin{enumerate}
\item The algorithm \emph{converges to the exact solution} $u(T_N)$ and not to the solution achieved by the fixed chosen fine solver $\cF(0,T_N,u(0))$ \ym{(indeed, for any $N$, $\zeta_k^N\longrightarrow 0$ as $k\longrightarrow\infty$)}.
\item For a final target accuracy $\eta$, the parallel efficiency will always be superior to the classical approach (see section \ref{sec:speedup}).
\item \corr{We minimize the computational ressources (degrees of freedom)} because we identify the minimal required accuracies at each iteration (see equation \eqref{eq:zeta}). Early iterations use a loose tolerance, thus avoiding unnecessary work due to oversolving, while later iterations use tighter tolerances to deliver accuracy.
\item The dynamical refinements of the fine solver invite to incorporate adaptive solvers with a posteriori error estimators to the parareal scheme. 
\end{enumerate}

\section{\om{Parallel efficiency}}
\label{sec:speedup}
It is difficult to give accurate a priori estimations for the speed-up and efficiency of the method due to its adaptive nature so the actual performance can only be established through relevant examples. In section \ref{sec:numerics}, we give some results for the case of the Brusselator system. Despite this difficulty in estimation, we make some general remarks in this section, which aim primarily at highlighting the relevance of the cost of the coarse solver. The speed-up is defined as the ratio
\begin{equation}
\speedup_{\AP / \seq}(\eta,[0,T])
\coloneqq
\frac{\cost_{\seq}(\eta,[0,T])}{\cost_\AP(\eta,[0,T])}
\label{eq:speed-up}
\end{equation}
between the cost to run a sequential fine solver achieving a target accuracy $\eta$ with the cost to run an adaptive parareal algorithm providing at the end the same target accuracy $\eta$. 
The parallel efficiency of the method is then defined as the ratio of the above speed up with the number of processor which gives a target of 1 to any parallel solver:
\begin{equation*}
\eff_{\AP / \seq}(\eta,[0,T])
\coloneqq
\frac {\speedup_{\AP / \seq}(\eta,[0,T])}{ \NI}.
\end{equation*}

Assume that $g_k^N$ and $f_k^N$ are the numerical costs to realize $\cG(T_N,\Delta T,y^N_k)$ and $[\cE(T_N,\Delta T,y^N_k),\zeta^N_k]$. Since the tolerances $\zeta^N_k$ decrease with $k$, we have $f_0^N< \dots <  f_{K(\eta)}^N$. Neglecting the communication delays, the cost of the adaptive solver is
$$
\cost_\AP(\eta,[0,T]) = \sum_{k=0}^{K(\eta)} \sum_{N=0}^{\NI-1} g_k^N + \sum_{k=0}^{K(\eta)-1} \sum_{N=0}^{\NI-1} f_{k}^N.
$$
The classical parareal algorithm involves, at every iteration $k\in \{0,...,K(\eta)\}$ propagations at the highest accuracy $\zeta^N_{K(\eta)}$. Thus its cost is
$$
\cost_\CP(\eta,[0,T]) = \sum_{k=0}^{K(\eta)} \sum_{N=0}^{\NI-1} g_k^N + K(\eta) \sum_{N=0}^{\NI-1} f_{K(\eta)}^N,
$$
from which it directly follows that \ym{(at least if, with obvious notation, $K_{AP}(\eta) = K_{CP}(\eta)$)}
$$
\eff_{\AP / \seq}(\eta,[0,T]) > \eff_{\CP / \seq}(\eta,[0,T]).
$$
Thus the parallel performance of the adaptive algorithm is at least the one of the classical version. Note that this holds even when communications are not negligible since there is the same amount of information exchange in both algorithms.

We next give a more quantitative statement on an admittedly idealized setting. Assume that the cost of the coarse solve is negligible, that there is no communication delay and that the cost to realize $[\cE(T_N,\Delta T,y^N_k),\zeta^N_k]$ is
\begin{equation}
\label{eq:costModel}
f_k^N =  \Delta T (\zeta^N_k)^{-1/\alpha}.
\end{equation}
This assumption for the cost is, for instance, reasonable when we use an explicit time-stepping method of order  $\alpha>0$. It would also hold for an implicit method where a direct solver can be used. Note that $\alpha$ could actually depend on $k$ but we stick to this simple model for clarity of exposition.

\begin{proposition} 
\label{prop:ideal}
If $f_k^N = \Delta T (\zeta^N_k)^{-1/\alpha},$ for some $\alpha>0$ and if the cost of the coarse solver is negligible with respect to $f_k^N$ for any $k\geq0$, then
\begin{equation*}
\eff_{\AP / \seq}(\eta,[0,T])=\frac{1-\tau^{1/\alpha}}{1-\tau^{K(\eta)/\alpha}} \sim \frac{1}{(1+\vepsdgs^{1/\alpha})}.
\end{equation*}
Therefore
\begin{equation*}
\speedup_{\AP / \seq}(\eta,[0,T]) \sim   \NI \frac{1}{(1+\vepsdgs^{1/\alpha})}.
\end{equation*}
\end{proposition}

\begin{proof}
The cost of the scalable adaptive parareal scheme after $K(\eta)$ iterations is
\begin{equation}
\cost_{\AP}(\eta,[0,T])=\Delta T \sum_{k=0}^{K(\eta)-1} \zeta_k^{-1/\alpha} 
\end{equation}

Since we are in a range where the scheme converges, the quantity $\max_{0\leq N\leq \NI}\Vert y^N_k\Vert$ is bounded and thus there exists $0<\underline c\leq 1 \leq \overline c$  such that $\underline c   \leq \nu_k \leq \overline c $ for all $k\geq0$. We will account for this with the notation $\nu_k\sim 1$. Note that in fact $\underline c$ and $\overline c$ are close to one.
Let us start with the simple case $\alpha=1$ and denote $\underline K = K(\eta)-1$:
\begin{equation*}
\cost_{\AP}(\eta,[0,T])=\Delta T \sum_{k=0}^{\underline K} \zeta_k^{-1} =\Delta T \zeta_{K(\eta)-1}^{-1} \left(1 + \frac{\vepsdg}{\underline K} +\frac{\vepsdgs^2}{\underline K(\underline K-1)} +\dots + \frac{\vepsdgs^{\underline K-1}}{\underline K!} \right)
\end{equation*}
we thus derive
\begin{equation*}
\cost_{\AP}(\eta,[0,T])\le \Delta T  {\zeta_{K(\eta)-1}^{-1}} (1+\vepsdg)
\end{equation*}
In the general case ($\alpha>1$), the same is true with  
\begin{equation*}
\cost_{\AP}(\eta,[0,T])\le \Delta T  {\zeta_{K(\eta)-1}^{-1/\alpha}} (1+ \vepsdgs^{1/\alpha})
\end{equation*}
Up to this last factor, the current conclusion is that the global cost of the parareal procedure is equal to the last fine solver on each sub-interval with size $\Delta T$ (both the coarse and the previous fine propagations are negligible).

Since the accuracy that is obtained at the end of the parareal procedure (see (\ref{eq:convFadaptBoundbis})) is of the same order as the accuracy provided with classical parareal solver (\ym{compare with} (\ref{eq:convPlainBound})), it follows that
if we now take the last target accuracy $\zeta_{K(\eta)-1}^{-1/\alpha}$ of the adaptive algorithm as the accuracy of the fine scheme in the classical parareal algorithm, the cost would be
\begin{equation*}
\cost_\CP(\eta,[0,T])= K(\eta) \Delta T  {\zeta_{K(\eta)-1}^{-1/\alpha}},
\end{equation*}
Therefore,
\begin{equation}
\label{eq:ap-cp}
\frac{\cost_\AP(\eta,[0,T])}{\cost_{\CP}(\eta,[0,T])}
\sim \frac{1}{K(\eta)} (1+\vepsdgs^{1/\alpha}).
\end{equation}
In addition, we know that when the cost of the coarse solver is negligible,
\begin{equation}
\label{eq:sp-cp}
\speedup_{\CP/\seq}(\eta,[0,T])
=\frac{\cost_{\seq}(\eta,[0,T])  }{ \cost_\CP(\eta,[0,T]) }
= \frac{\NI}{K(\eta)},
\end{equation}
Dividing  \eqref{eq:sp-cp} by \eqref{eq:ap-cp} yields
\begin{equation*}
\speedup_{\AP / \seq}(\eta,[0,T]) \sim \NI \frac{1}{(1+\vepsdgs^{1/\alpha})}
\end{equation*}
and
\begin{equation*}
\eff_{\AP / \seq}(\eta,[0,T]) \sim \frac{1}{(1+\vepsdgs^{1/\alpha})}.
\end{equation*}
\end{proof}
\om{In the ideal setting of Proposition \ref{prop:ideal}:
\begin{itemize}
\item The parallel efficiency of the adaptive parareal algorithm does not depend on \emph{the final number of iterations}. This is in contrast to the classical version whose efficiency decreases with the final number of iterations $K(\eta)$ as $1/K(\eta)$.
\item The efficiency behaves like $1- o(\vepsdg)$ in the adaptive version, and $o(\vepsdg)$ rapidly goes to zero with $\vepsdg$. As soon as $\vepsdg$ becomes negligible \ym{with respect to } 1, we will be in the range of full scalability.
\end{itemize}
We emphasize that, obviously, the above idealized setting will never hold in practice, but the result is interesting in its own right since it highlights that the cost of the fine solver is no longer the main obstacle for full scalability in the adaptive setting: the cost of the coarse solver becomes now the major obstruction towards full efficiency.}

\section{\om{Numerical tests}}
\label{sec:numerics}

\subsection{Guidelines for a practical implementation}
\paragraph*{Practical choice of $\zeta_k^N$} Formula \eqref{eq:zeta} of the convergence analysis of section \ref{sec:feasible} gives an estimate for  $\zeta_k^N$ that one could in principle use for the implementation. However, these tolerances may not be optimal because they are derived from a theoretical convergence analysis based on abstract conditions for the coarse and fine solvers. This was confirmed during our numerical tests where we observed that using estimates \eqref{eq:zeta} for $\zeta_k^N$ did not deliver satisfactory \ym{enough} results. This is the reason why it is necessary to devise a practical rule to set $\zeta_k^N$. We have explored the following choice: if $\eta$ is the final target accuracy, the classical parareal algorithm is usually run with a solver that delivers a slightly higher accuracy, say $\eta/2$. Assume that the classical algorithm converges in $K_{\CP}(\eta)=K$ iterations. We propose to build the tolerances of $\zeta_k^N$ in such a way \ym{to target} that $K_{\AP}(\eta) = K_{\CP}(\eta)$ and such that the cost of the last fine propagation is of the order of the sum of the previous ones. This motivates to set
$$
\zeta_k^N =
\begin{cases}
 \eps_{\cG}^{1-\frac{k+1}{K}} \left(\frac{\eta}{2}\right)^{\frac{k+1}{K}},&\quad \text{if } k<K \\
 \eta/2,&\quad \text{if } k\geq K.
 \end{cases}
$$
The numerical example of the next section uses these tolerances.

\paragraph*{Load balancing:} For simplicity of exposition, the algorithm has so far been discussed  for $\NI$ subintervals of uniform size $\Delta T$. However, this decomposition may lead to a task imbalance because some time intervals may have more complex dynamics than others, requiring more degrees of freedom, thus more computational time. In order to balance tasks as efficiently as possible, we dynamically adapt the size of the $\NI$ subintervals in a way to have the fine solver propagations as balanced as possible among processors.

\subsection{Results for several stiff ODEs}
 \rev{
We apply our adaptive algorithm to several stiff ODEs where the only mecanism for adaptivity is time. Our results illustrate that our approach improves the speed-up and efficiency with respect to the classical non-adaptive parareal method. We also show that the main element affecting performance is no longer the cost of the fine solver but the cost of the coarse solver. In extreme cases, this cost may even prevent any speed-up at all (see section \ref{sec:extreme}) and puts this obstruction at the forefront for future research. The code to reproduce the numerical results is available online at:
\begin{center}
\href{https://plmlab.math.cnrs.fr/mulahernandez/parareal-adaptive}{https://plmlab.math.cnrs.fr/mulahernandez/parareal-adaptive}
\end{center}
Other ODEs can easily be tested as indicated in the instructions.
}

\rev{Note that the algorithm could also be applied to PDEs but we defer the presentation of numerical examples to future works since this requires full space-time adaptive techniques which are a topic in itself since they are challenging to formulate and deploy and very specific to each type of problem.}

\subsubsection{The Brusselator system}

We consider the brusselator system
\begin{align*}
\begin{cases}
x' = A+x^2y-(B+1)x \\
y' =Bx-x^2y,
\end{cases}
\end{align*}
with initial condition $x(0)=0$ and $y(0)=1$. This is a stiff ODE that models a chain of chemical reactions. It was already studied in a previous work on the parareal algorithm (see \cite{GH2008}). The system has a fixed point at $x=A$ and $y=B/A$ which becomes unstable when $B>1+A^2$ and leads to oscillations. We place ourselves in this oscillatory regime by setting $A=1$ and $B=3$. The dynamics present large velocity variations in some time subintervals, making the use of adaptive time-stepping schemes particularly desirable for an appropriate treatment of the transient.

For the coarse solver, we set $$\eps_{\cG}=0.1,$$ and use an explicit Runge Kutta method of order 5 with an adaptive time-stepping (see \cite{DP1980}). 
For the fine solver, we use the implicit Runge-Kutta method of the Radau IIA family of order 5 with adaptive time-stepping (see \cite{HW1996,HW1999}). Both integrators are available in the ODE integration library of Scipy\footnote{\url{https://docs.scipy.org/doc/scipy/reference/generated/scipy.integrate.solve_ivp.html}} which we have used \rev{in our library.}

As already discussed, the target accuracies $\zeta_k^N$ should be ensured by rigorous a posteriori error estimators. \corr{However, these type of estimators are unfortunately not available in the Scipy library and we are not aware of any mainstream library with this capability. As a surrogate, we have used the above mentioned classical ODE integrators that only guarantee \emph{local} accuracy between time-steps $t_n \to t_{n+1}$, but not \emph{global} accuracy between macro intervals $[T_N, T_{N+1}]$ (composed of several time-steps). The local accuracy can be specified in the library routine via the parameters \texttt{atol} and \texttt{rtol} of the function \texttt{scipy.integrate.solve_ivp}. To relate this local accuracy control to the global one, we have built a priori a ``chart'' mapping accuracies of the solver on macro-intervals against the tolerance parameters \texttt{atol} and \texttt{rtol} of the library. To simplify, these two parameters have been set to be equal ($\texttt{atol}= \texttt{rtol}$) and their value is fixed according to the chart. As an example, we provide a chart for $T=20$ for the scheme of the fine solver in Figure \ref{fig:chart}. The dots are computed values: for a given value of the parameter $\texttt{atol}$, we examine the accuracy $\eps$ of the solver. We then interpolate the points with a cubic spline interpolation. This way, for a given intermediate accuracy $\zeta^N_k$ in the parareal algorithm, we can easily adapt the parameter value $\texttt{atol}$ that is required.}

\begin{figure}
\begin{center}
\includegraphics[scale=0.4]{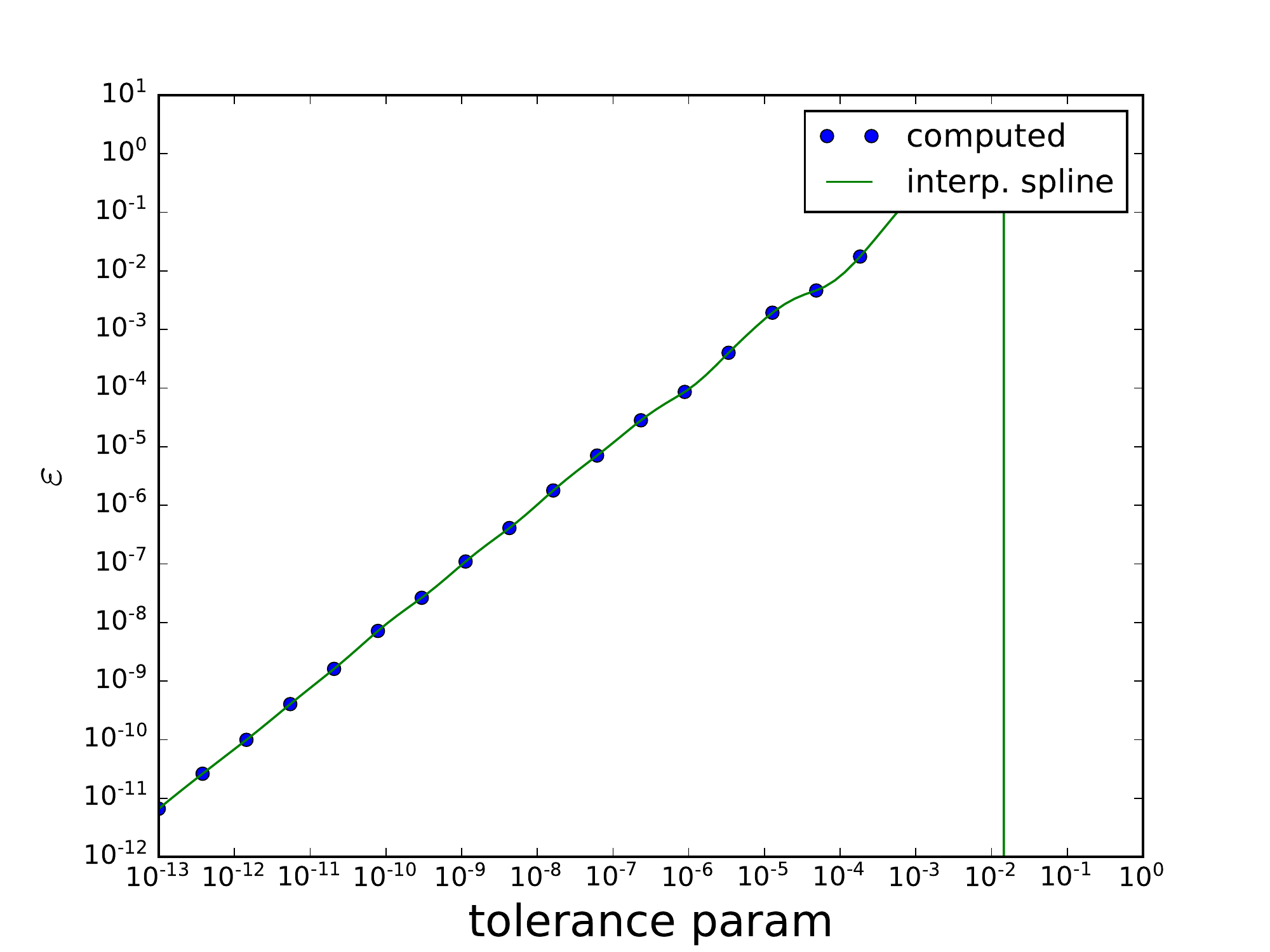}
\caption{Mapping of the accuracies $\varepsilon$ against the tolerance parameters ($\texttt{atol}= \texttt{rtol}$) of the library. The dots are computed values: for a given value of the parameter, we examine the accuracy $\eps$ of the solver. We then interpolate the points with a cubic spline interpolation. This way, for a given intermediate accuracy $\zeta^N_k$ in the algorithm, we can infer the parameter value $\texttt{atol}$ and $\texttt{rtol}$.  Case $T=20$, integrator of the fine solver.}
\label{fig:chart}
\end{center}
\end{figure}

We use formula \eqref{eq:speed-up} to compare the speed-up of the classical and adaptive parareal algorithm in terms of the number of \ym{ operations} involved in the numerical solution (communication delays have not been taken into account). For the costs $g^N_k$ and $f^N_k$, we take into account:
\begin{itemize}
\item the number of time steps (which is adaptively increased as we tightned the accuracy),
\item the number of right-hand side evaluations,
\item for the fine solver, we additionally count the number of evaluations of the Jacobian matrix and of the number of linear system inversions.
\end{itemize}
In Figure \ref{fig:speedup}, we plot the obtained speed-up for different configurations:
\begin{itemize}
\item the final time $T$ varies from $100$ to $900$,
\item the final target accuracy is $\eta=10^{-6}$ or $\eta=10^{-8}$,
\item the number of processors $\NI$ varies from 10 to 100.
\end{itemize}
As anticipated in section \ref{sec:speedup}, the speed-up of the adaptive parareal is always superior to the one of the classical parareal. We observe that the gain is marginal for a moderate accuracy ($\eta=10^{-6}$) but it is about 2.5 times larger for $\eta=10^{-8}$. Note that sometimes the speed-up does not increase monotonically as the number of processors $\NI$ increases. Also, the speed-up generally increases with $\NI$ but the increase is rather moderate.

\begin{figure}
    \centering
    \begin{subfigure}[b]{0.45\textwidth}
        \includegraphics[width=\textwidth]{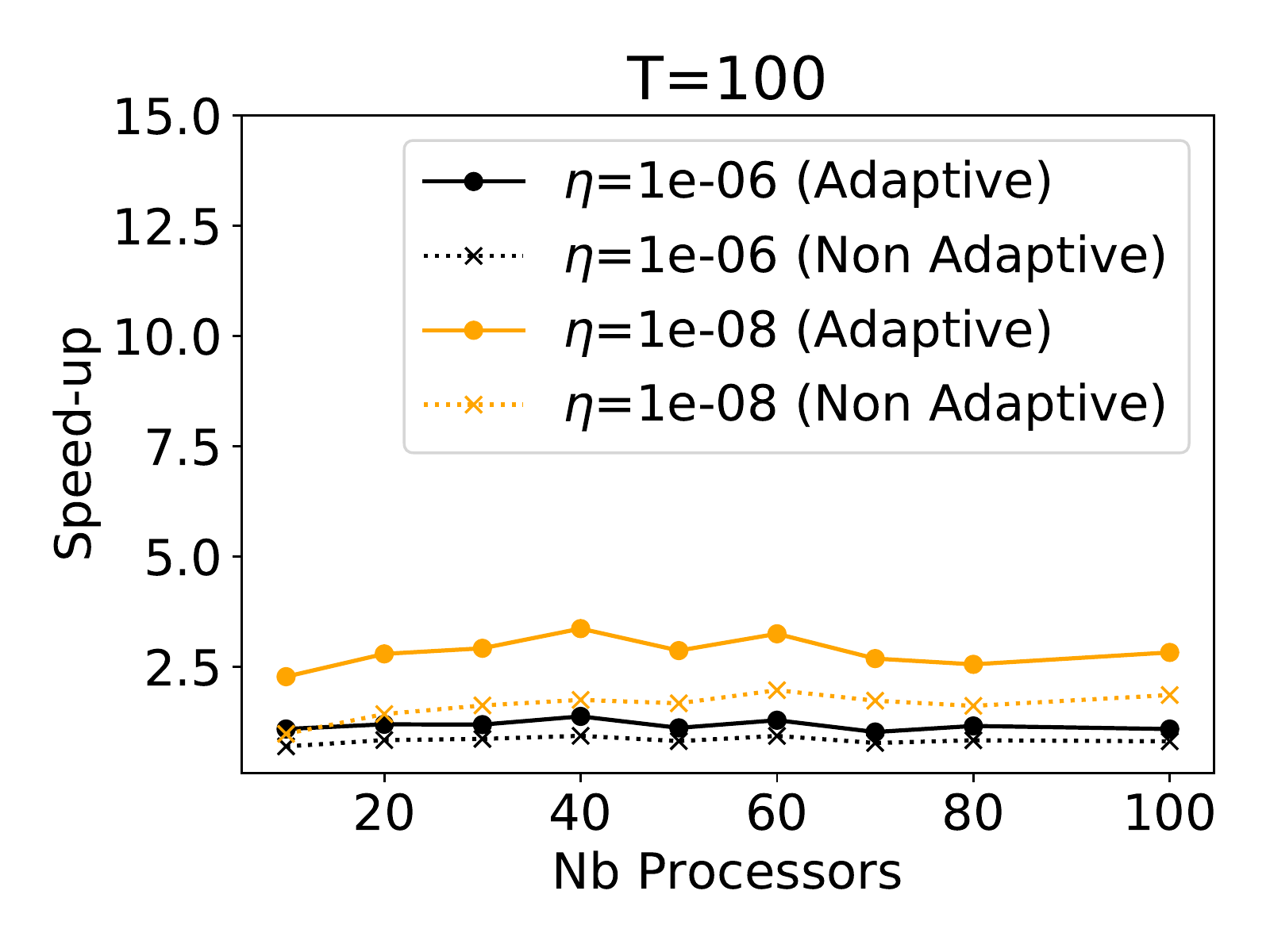}
        \label{fig:spT100}
    \end{subfigure}
    ~ 
    \begin{subfigure}[b]{0.45\textwidth}
        \includegraphics[width=\textwidth]{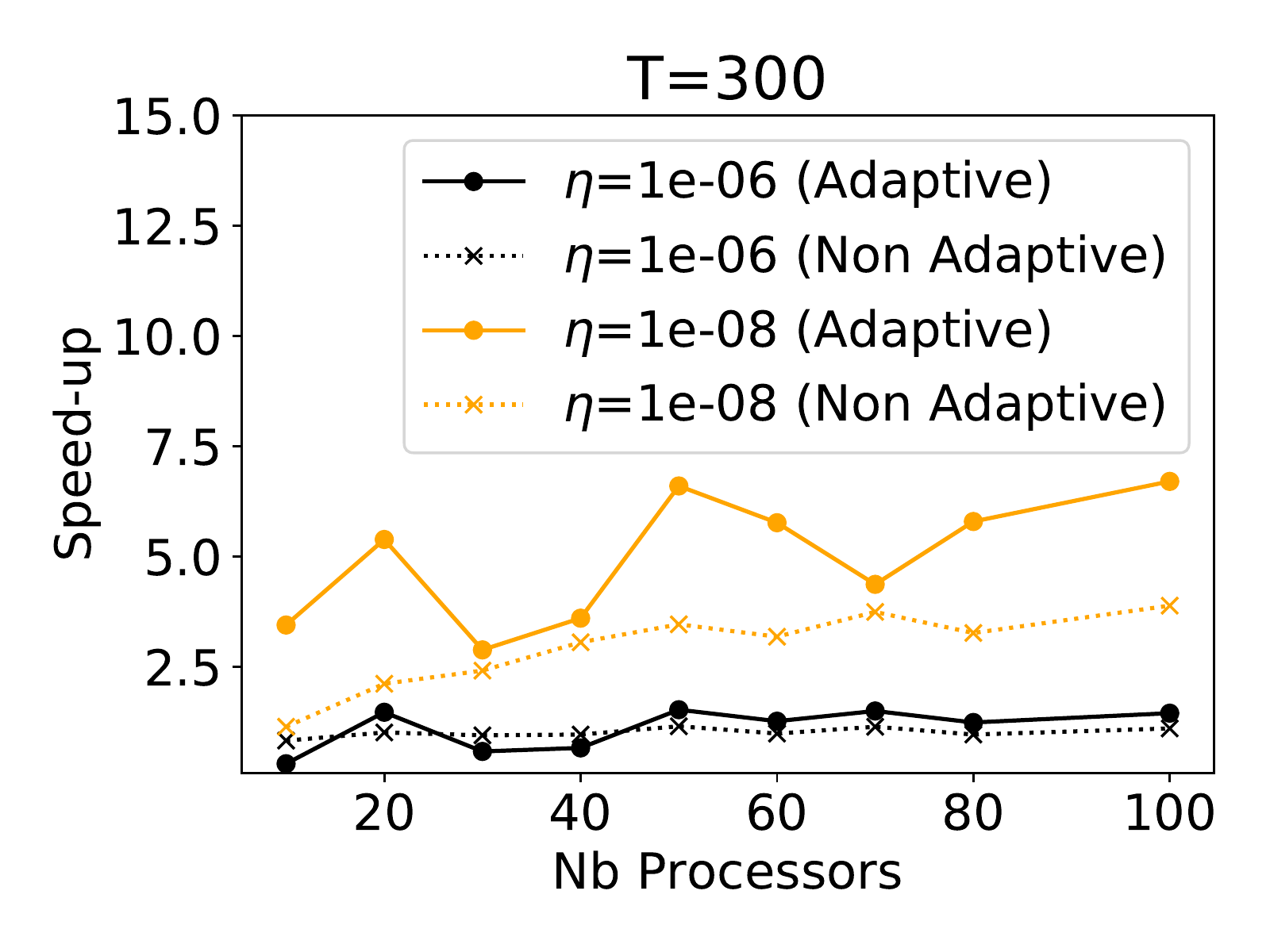}
        \label{fig:spT300}
    \end{subfigure}
    
    \begin{subfigure}[b]{0.45\textwidth}
        \includegraphics[width=\textwidth]{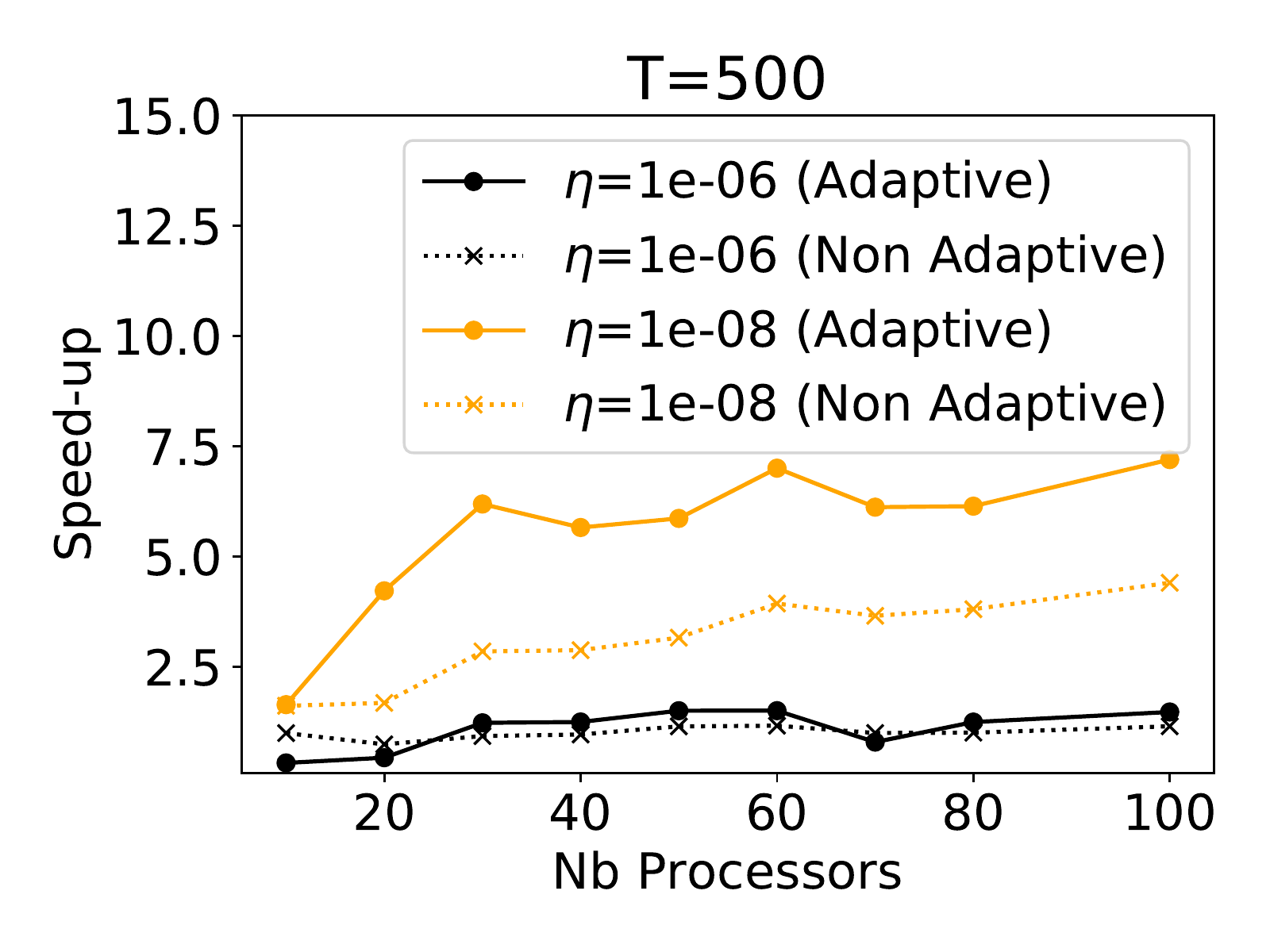}
        \label{fig:spT500}
    \end{subfigure}
    ~
    \begin{subfigure}[b]{0.45\textwidth}
        \includegraphics[width=\textwidth]{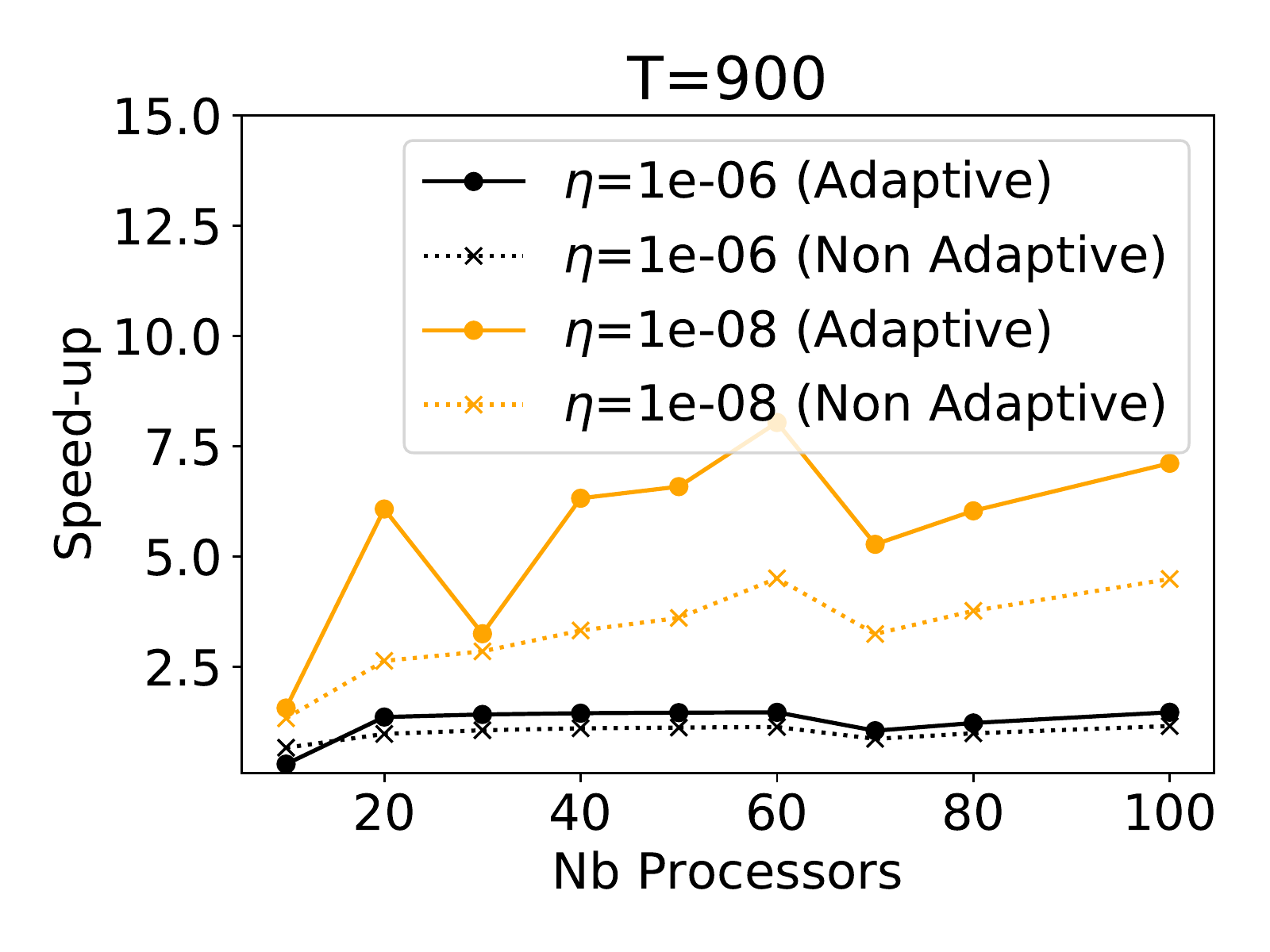}
        \label{fig:spT900}
    \end{subfigure}
    \caption{Speed-up \corr{in comparison to running a sequential fine solver} as a function of the number of processors $\NI$. Dashed lines: classical parareal. Continuous lines: Adaptive parareal.}\label{fig:speedup}
\end{figure}

\begin{figure}
\begin{center}
\includegraphics[width=0.45\textwidth]{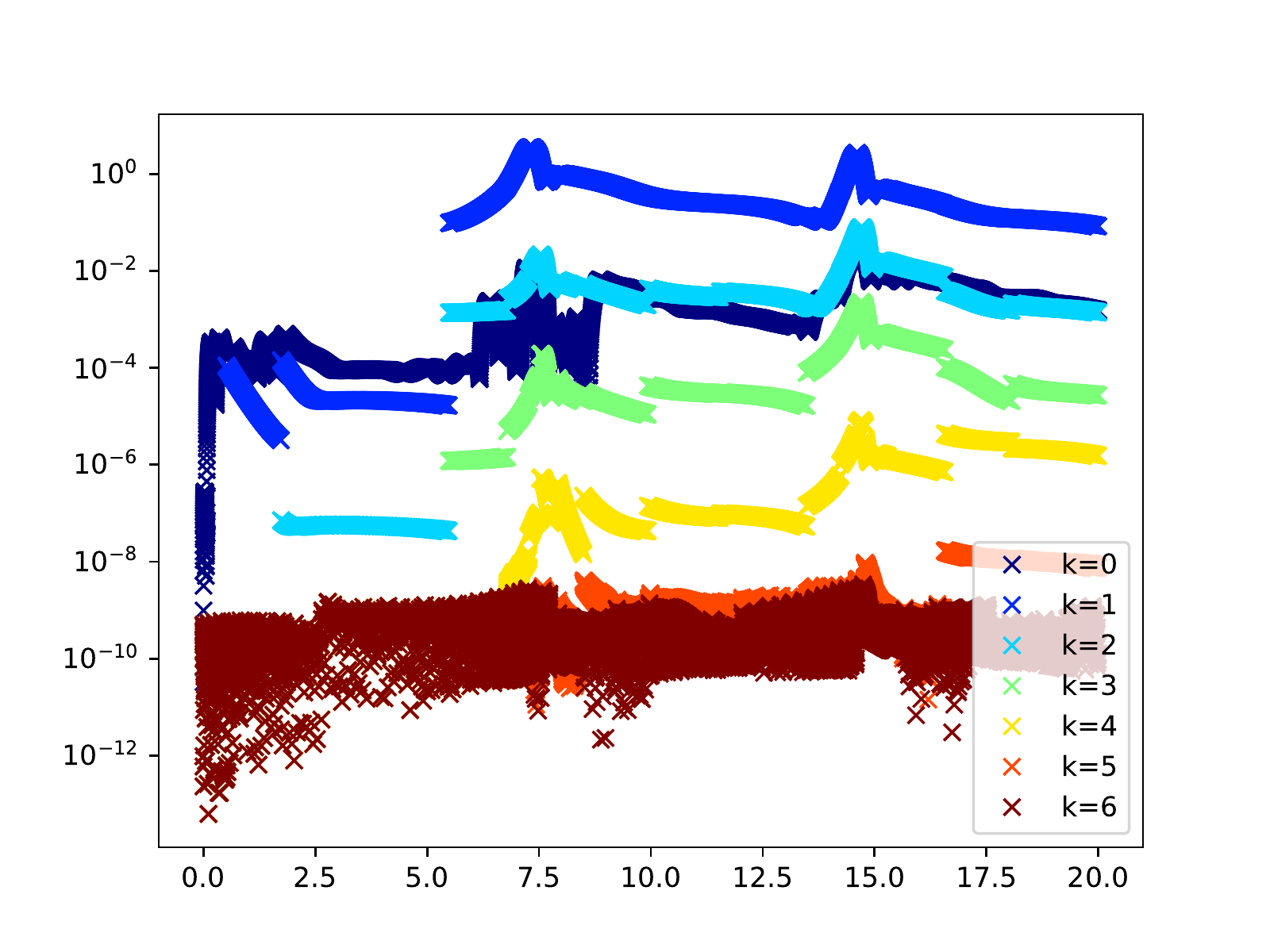}
\includegraphics[width=0.45\textwidth]{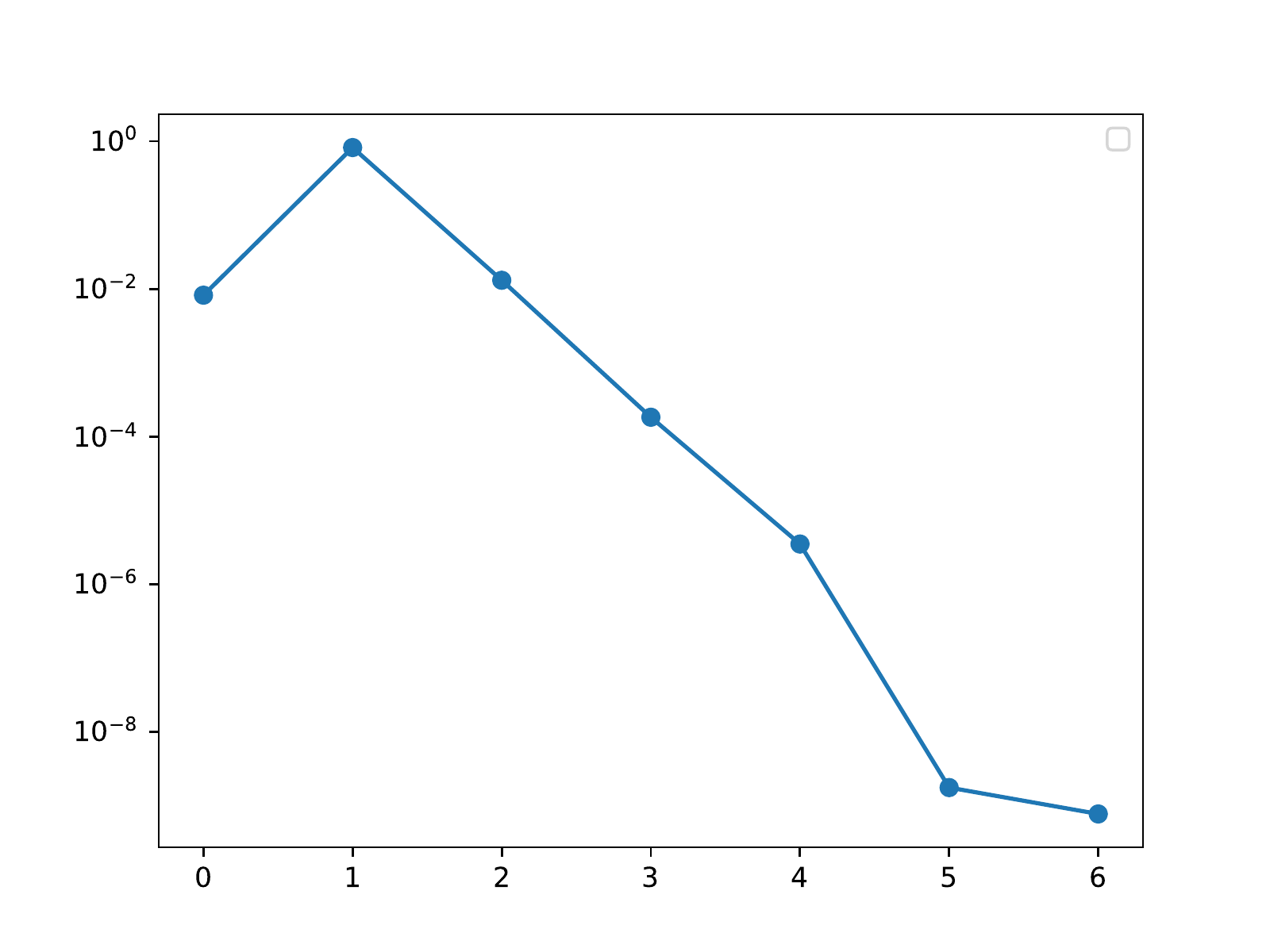}
\\
\includegraphics[width=0.45\textwidth]{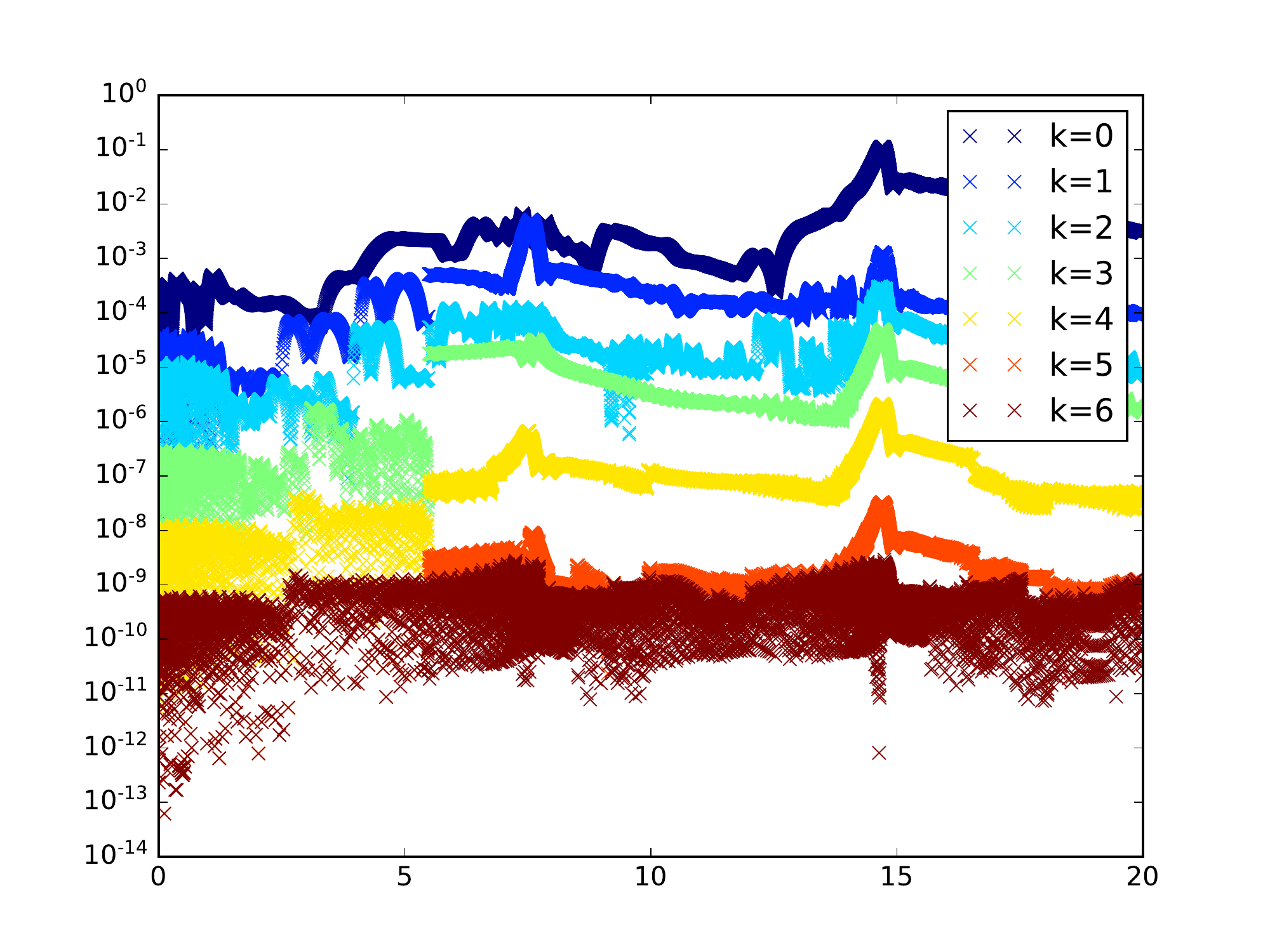}
\includegraphics[width=0.45\textwidth]{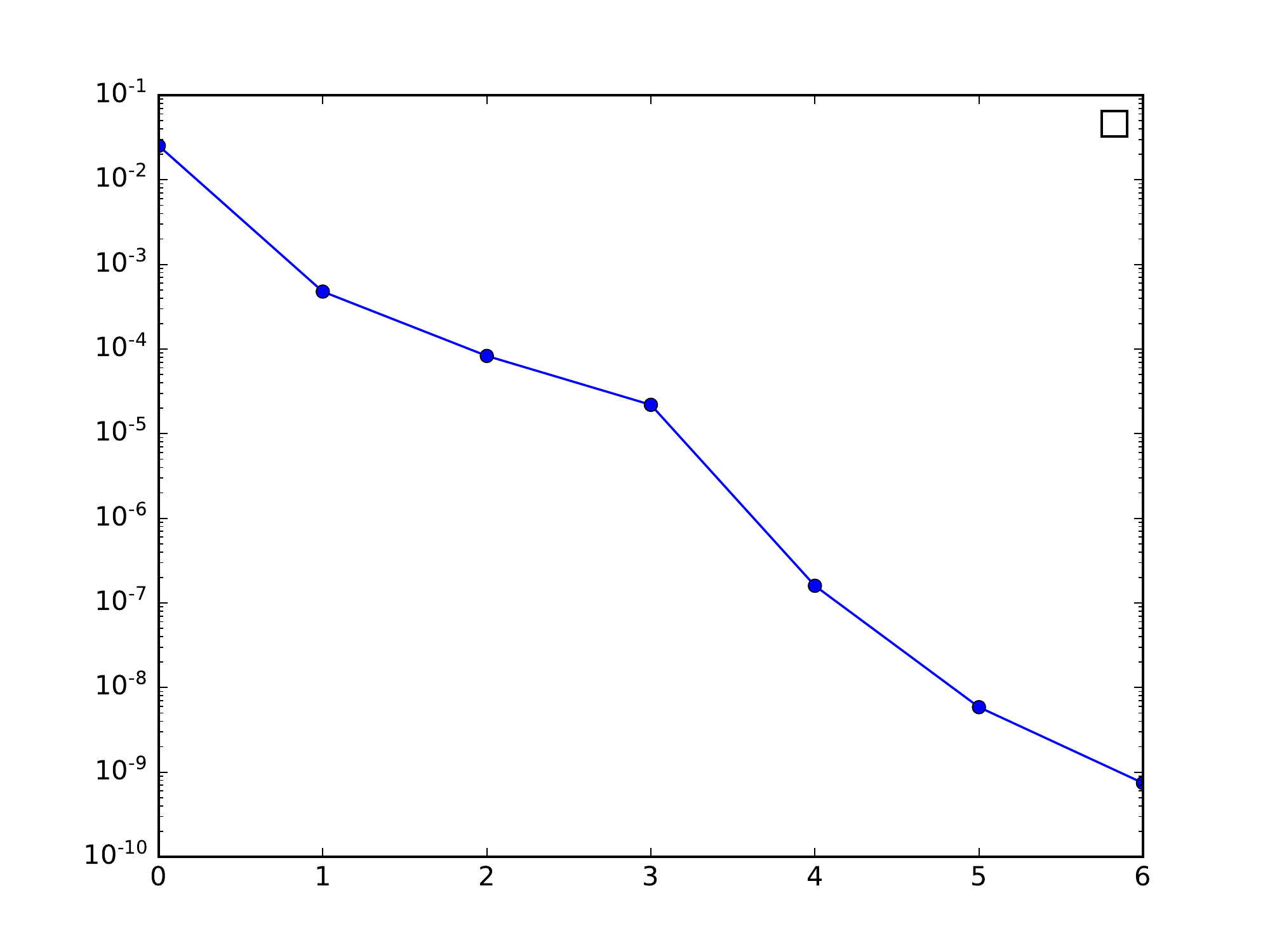}
\end{center}
\caption{\textbf{Brusselator:} Convergence history of the errors \corr{for $T=\corr{20}$, $\eta = 10^{-8}$ and $\NI=\corr{20}$}. Top: classical parareal. Bottom: adaptive parareal.
Left: errors  \corr{of the fine solver at every fine time-step. Right: maximum parareal error at each iteration $k$}.}
\label{fig:conv}
\end{figure}

The values significantly differ from the range of full scalability and we next explain why this is mainly due to the cost of the coarse solver. Since the problem is stiff and we consider relatively long time intervals, it has been necessary to use a sufficiently accurate coarse solver. This explains our choice of an explicit Runge-Kutta scheme of order 5. To illustrate the impact of its cost, let us fix $T=500$, $\eta = 10^{-8}$ and $\NI=50$ (other parameters would yield similar conclusions). We compare the speed-up and efficiency when we count or do not count the cost of the coarse solver in Table \ref{tab:sp}. Obviously, when we do not count the cost of the coarse solver, the performance of both algorithms improves but it is particularly increased in the case of the adaptive version. If the cost of $\cG$ was negligible, it would deliver a very satisfactory efficiency of $75.52\%$. This is five times larger than what the classical parareal would yield. This analysis illustrates that the major obstacle to achieve competitive scalabilities is no longer the cost of the fine solver like in the classical version, but the cost of the coarse propagator.

\begin{table}
\begin{center}
\begin{tabular}{ | c | c | c |}
  \hline			
   \textbf{Speed-up} &  Classical parareal  & Adaptive Parareal \\
  \hline
  With cost $\cG$ & 4.06 & 7.38 \\
  \hline
  Without cost $\cG$ & 7.38 & 37.76 \\
  \hline  
\end{tabular}
\end{center}
~
\begin{center}
\begin{tabular}{ | c | c | c |}
  \hline			
   \textbf{Efficiency} &  Classical parareal  & Adaptive Parareal \\
  \hline
  With cost $\cG$ & 8\% & 14.76\% \\
  \hline
  Without cost $\cG$ & 14.76\%  & 75.52\% \\
  \hline  
\end{tabular}
\end{center}
\caption{\textbf{Brusselator:} Impact of the cost of the coarse solver. Speed-up and efficiency with \corr{$T=500$, $\eta = 10^{-8}$ and $\NI=50$}.}
\label{tab:sp}
\end{table}

We next give some insight on the differences in the convergence behavior of both algorithms. We fix $T=\corr{20}$, $\eta = 10^{-8}$ and $\NI=\corr{20}$ and plot in Figure \ref{fig:conv} the convergence history of the parareal solution in terms of:
\begin{itemize}
\item the errors of the fine solver at every fine time-step
\item the maximum error of the parareal solution at the macro-intervals
$$
\max_N \Vert u(T_N)-y^N_k \Vert
$$
\end{itemize}
Note that the maximum error in the adaptive scheme steadily decreases to the desired accuracy whereas the error in the classical scheme degrades at iteration $k=1$ before converging. This type of behavior has been observed for all other configurations and we conjecture that an important difference in accuracy between the coarse and the fine solver at early stages of the algorithm may be the cause. Finally, an inspection of the error of the fine solver shows that the adaptive algorithm succeeds to reduce the error at every time $t$ in a much more uniform way than the classical algorithm.

\subsubsection{\rev{The Van der Pol oscillator}}
We next consider the Van der Pol oscillator
\begin{align*}
\begin{cases}
x' = y \\
y' =\mu(1-x^2)y -x,
\end{cases}
\end{align*}
with initial condition $x(0)=2$ and $y(0)=0$. When $\mu=0$, this equation is a simple nonstiff harmonic oscillator. When $\mu>0$, the system has a limit cycle and becomes stiffer and stiffer as its value is increased. For our tests, we set $\mu =4$ which is already a relatively stiff case.

Like in the example of the Brusselator system, we set $\eps_{\cG}=0.1$ for the coarse solver and use an explicit Runge Kutta method of order 5 with an adaptive time-stepping (see \cite{DP1980}). For the fine solver, we use the implicit Runge-Kutta method of the Radau IIA family of order 5 with adaptive time-stepping.

In Figure \ref{fig:speedup-vdp}, we plot the obtained speed-up for different configurations:
\begin{itemize}
\item the final time $T$ is $1000$ or $2000$,
\item the final target accuracy is $\eta=10^{-6}$ or $\eta=10^{-8}$,
\item the number of processors $\NI$ varies from 10 to 100.
\end{itemize}
Like in the previous example, the adaptive algorithm outperforms the nonadaptive version in terms of speed-up. However, the gain is marginal for moderate accuracies $\eta=10^{-6}$. For high accuracy $\eta=10^{-8}$, the adaptive algorithm improves the speed-up by a factor of about 2 to 3 times with respect to the classical one. The improvement is more significant for large $T$.

\begin{figure}[h!]
    \centering
    \begin{subfigure}[b]{0.45\textwidth}
        \includegraphics[width=\textwidth]{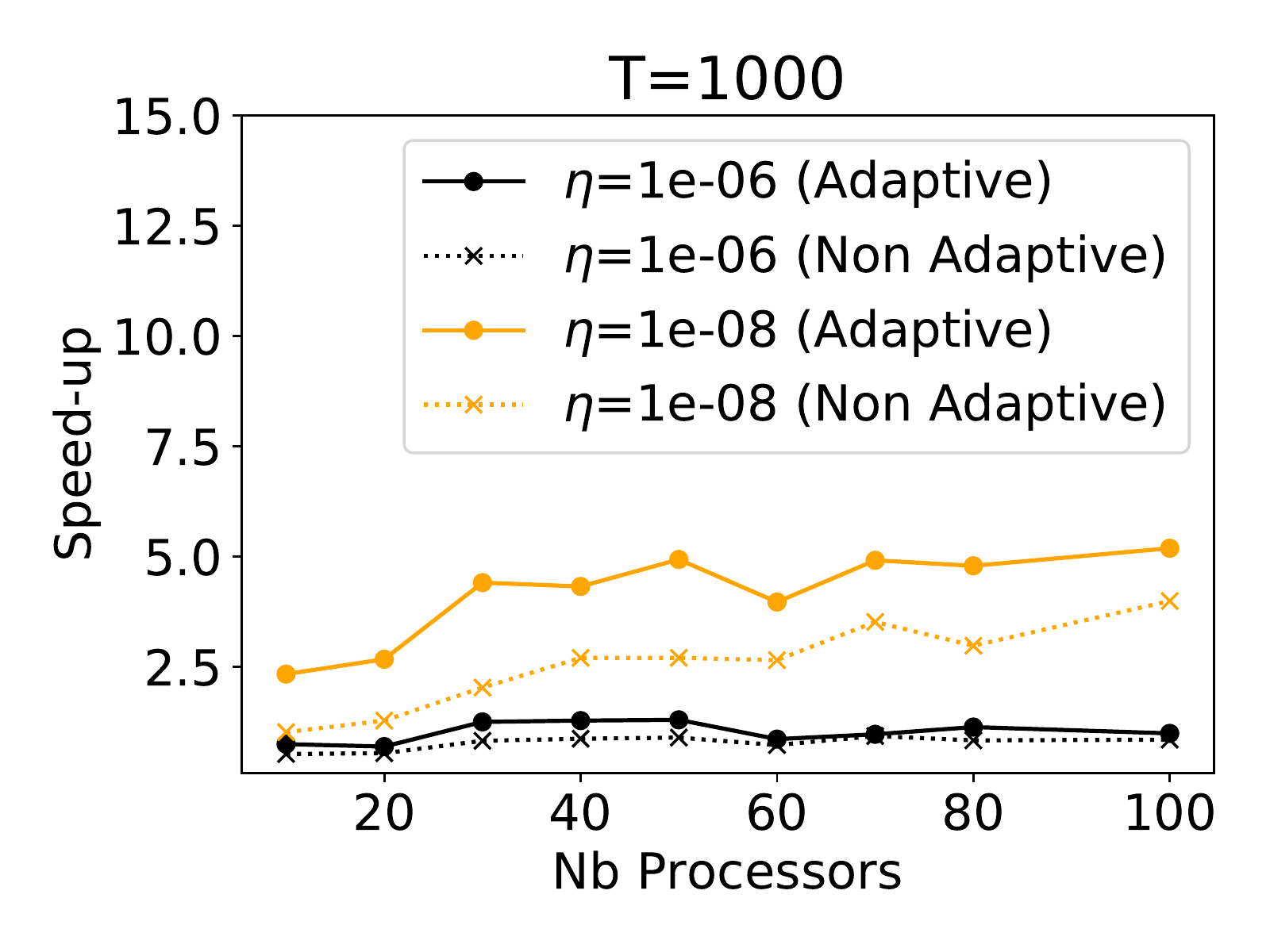}
        \label{fig:spT1000}
    \end{subfigure}
    ~ 
    \begin{subfigure}[b]{0.45\textwidth}
        \includegraphics[width=\textwidth]{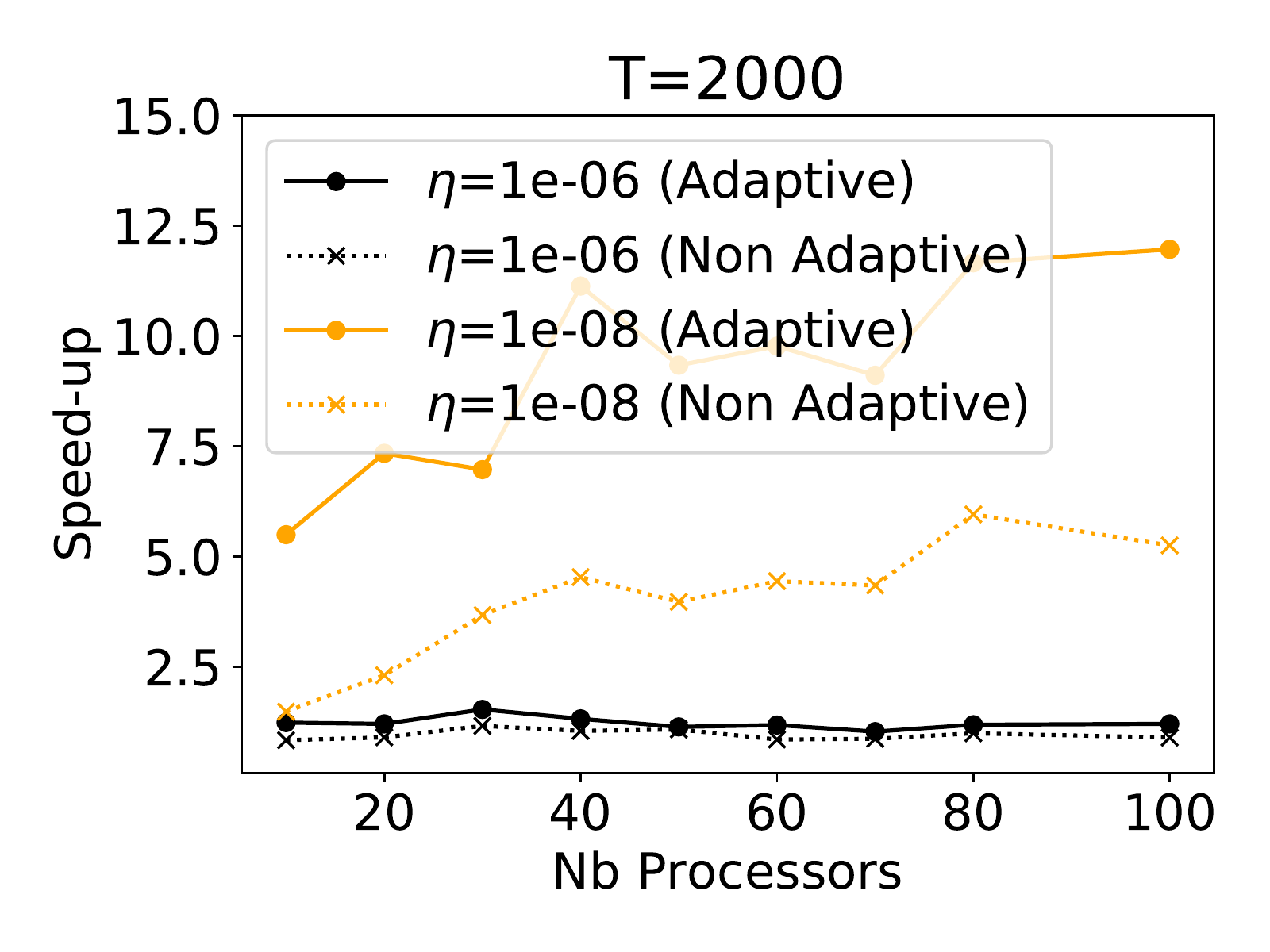}
        \label{fig:spT2000}
    \end{subfigure}
    \caption{\textbf{Van der Pol:} speed-up \corr{in comparison to running a sequential fine solver} as a function of the number of processors $\NI$. Dashed lines: classical parareal. Continuous lines: Adaptive parareal.}
    \label{fig:speedup-vdp}
\end{figure}

In Table \ref{tab:sp-vdp}, we illustrate that the coarse solver is again the main bottleneck to reach high parallel efficiency in the adaptive algorithm: we examine the speed-up and efficiency for $T=2000$, $\eta = 10^{-8}$ and $\NI=40$ when we take and do not take into account the cost of the coarse solver.

\begin{table}
\begin{center}
\begin{tabular}{ | c | c | c |}
  \hline			
   \textbf{Speed-up} &  Classical parareal  & Adaptive Parareal \\
  \hline
  With cost $\cG$ & 4.54  & 11.14 \\
  \hline
  Without cost $\cG$ & 6.61 & 32.63 \\
  \hline  
\end{tabular}
\end{center}
~
\begin{center}
\begin{tabular}{ | c | c | c |}
  \hline			
   \textbf{Efficiency} &  Classical parareal  & Adaptive Parareal \\
  \hline
  With cost $\cG$ & 11.35\% & 27.8\% \\
  \hline
  Without cost $\cG$ & 16.5\%  & 81.56\% \\
  \hline  
\end{tabular}
\end{center}
\caption{\textbf{Van der Pol:} Impact of the cost of the coarse solver. Speed-up and efficiency with \corr{$T=2000$, $\eta = 10^{-8}$ and $\NI=40$}.}
\label{tab:sp-vdp}
\end{table}

\subsubsection{\rev{Discussion on extremely challenging cases of highly stiff ODEs: Oregonator and SEIR epidemic model}}
\label{sec:extreme}
In extreme cases, the cost of the coarse solver may prevent any speed-up at all (see section \ref{sec:extreme}) and puts this obstruction at the forefront for future research. A prominent case are highly stiff ODEs. In the provided code, we can observe this fact in the case of the Oregonator system of equations. In view of the ongoing pandemic of COVID-19 at the time when this article was written, we have also tested an SEIR epidemic model which has very recently been proposed in \cite{covid-model-2020} as a simple model for the spread of the virus in the Wuhan city area. The model captures the effect of the presence of individual reaction to the risk of infection and governmental action. If we run the parareal algorithm for the latter model, an integrator of type LSODA \cite{Petzold1983} for both coarse and fine solvers seems well adapted since it alternates between Adams or BDF integration for nonstiff and stiff parts and it has automatic stiffness detection. To reach a target accuracy $\eta=5.10^{-6}$, it has been necessary to set $\eps_{\cG} = 5.10^{-2}$, making the cost of the coarse solver too expensive to yield any parallel efficiency.  However, if we could find an inexpensive coarse solver, perhaps based on empirical data or on good back-of the enveloppe calculations, our adaptive parareal algorithm would yield interesting speeds-ups. For example, with $N=10$ processors, we would get a speed-up of $3.79$ (versus $1.4$) in the nonadaptive version.

\section{Conclusions and perspectives}
The new adaptive formulation of the parareal algorithm opens the door to improve significantly the parallel efficiency of the method provided that the cost of the coarse solver is moderate. The increasing target tolerances which have to be met at each step \ym{allows} to use online stopping criteria involving {\sl a posteriori} estimators. The developed methodology remains theoretical since we have not quantified the impact of communication delays between processors nor potential memory issues (note however that the load balancing is devised with the purpose of equilibrating tasks and memory). In the framework of the ANR project  ``Cin\'e-Para (ANR-15-CE23-0019)'', we are working on these issues that are of a different level of theory and involve different collaborators. This will be the topic of another paper. In addition to this, several extensions based on the current findings are subject of ongoing works, in particular, the coupling of the adaptive parareal with adaptive space-time schemes, the coupling of parareal with internal iterative solvers like in the discussion of section \ref{sec:APode} of the appendix, and also the development of inexpensive coarse solvers.

\section*{Acknowledgment}
This work was funded by the ANR project  ``Cin\'e-Para'' (ANR-15-CE23-0019).

\appendix

\section{Enriching the input information with previous iterations}
\label{sec:enrich-input}
Usually,  solvers to realize $[\cE(T_N,\Delta T,y^N_k),\zeta^N_k]$  are built using only
$\bI^N_k=\{ T_N, \Delta T, y^N_k\}$
as input information. We account for this idea with the notation
\begin{equation*}
\cS(T_N,\Delta T,y^N_k) \xrightarrow[(\bI^N_k,\cost^N_k)]{} [\cE(T_N,\Delta T,y^N_k);\zeta^N_k],
\end{equation*}
and the numerical cost, denoted $\cost^N_k$, increase as $\zeta^N_k$ is tightened. In this section, we discuss how to enhance the gain in efficiency of the adaptive algorithm by discussing ways to increase the accuracy of the solver $[\cE(T_N,\Delta T,y^N_k),\zeta^N_k]$  across the iterations while maintaining the cost to realize it as independent as possible from $\zeta^N_k$, $N$, and $k$. For this, one possibility is to enrich $\bI^N_k$ with data produced during the previous parareal iterations (although it would of course be at the cost of increasing the storage requirements).

Let $\bP^N_k$ denote the intermediate information that has been produced at iteration $k$ between $[T_N,T_{N+1}]$ and by $\bP_k \coloneqq \cup_{N=0}^{\NI-1}\bP^N_k$ all the information produced at step $k$. Using
$\widetilde \bI^N_k = \{ \bI_k^N, \bP^{N-1}_k,\dots, \bP^0_k, \bP_{k-1},\dots,\bP_0\}$,
as input information, the idea is to see whether it could be possible to find a solver $\cS$ such that
\begin{equation}
\label{eq:reduceComp}
\cS(T_N,\Delta T,y^N_k) \xrightarrow[(\widetilde \bI^N_k,\cost)]{} [\cE(T_N,\Delta T,y^N_k);\zeta^N_k]
\end{equation}
with a constant and small complexity $\cost$. Note that using the enriched set of information $\widetilde \bI^N_k$ means that we want to learn from the previous approximations of $u(T_{N+1})$ given by $[\cE(T_N,\Delta T,y^N_p),\zeta^N_p]$, $0\leq p\leq k-1$, to start the current algorithm closer to $u(T_{N+1})$. After each parareal iteration, we thus improve the accuracy, without increasing the work for solving because we start from a better input, accumulated from the previous parareal iterations.

In the rest of this section, we describe three relevant scenarios where we can approximate $\cE(T_N,\Delta T,y^N_k)$ by trying to build a scheme in the spirit of \eqref{eq:reduceComp}. The first two examples have already been presented in the literature and concern the coupling of parareal with spatial domain decomposition (section \ref{sec:ddm}) and with iterative high-order time integration schemes (section \ref{eq:sdc}). In these two cases, there is to date no complete convergence analysis since it remains to show that i) $\cE(T_N,\Delta T,y^N_k)$ is approximated with accuracy $\zeta_k$ and ii) the \cost~of the solver is really constant through the parareal steps. In addition to these two applications, we mention a third scenario where the convergence analysis can be fully proven. It concerns the solution of time-dependent problems involving internal iterative schemes at every time step. This idea was first analyzed in \cite{mulaPhD} in a restricted setting. \corr{It has also been applied in the framework of the MGRIT algorithm that couples parareal with multigrid iterative schemes (see \cite{FMOS2017})}. In section \ref{sec:APode}, we give the main setting and defer the analysis for a forthcoming paper.

\subsection{Parareal coupled with spatial domain decomposition}
\label{sec:ddm}
Here, we consider a solver $\cS=\DDM$ which involves spatial domain decomposition over $[T_N,T_{N+1}]$ ~\cite{guetatPhD, ABGM2017}. We assume that $\DDM$ involves a time discretization with \ym{a} small time step $\delta t < \Delta T$. Let $\nfi$ be the number of time steps on each interval $[T_N, T_{N+1}]$ so that we have the relations
$\Delta T =\nfi\delta t$ and $T = \NI \Delta T = \NI\ \nfi\delta t$ and the total number of time steps in $[0,T]$ is $\nFine \coloneqq \NI\ \nfi$.
The domain decomposition iterations act on a partition $\Omega=\cup_{l=1}^L \Omega_l$ of the domain. For $0\leq n\leq \nfi$, we denote by $u^{N,n,j}_k$ the solution produced by $\DDM$ at time $t=T_N+n\delta t$ after $j\geq0$ domain decomposition iterations. The notation $J^*$ will denote the last iteration (fixed according to some stopping criterion). At $j=0$, these iterations need to be initialized at the interfaces $\partial \Omega_l,\ 1\leq l \leq L$. The idea explored in, e.g.,~\cite{guetatPhD, ABGM2017}, is to take the values $u^{N,n,J^*}_{k-1}\vert_{\partial \Omega_l}$ at these interfaces as a starting guess for $0\leq n\leq \nfi$ so that
\begin{equation*}
\quad 
\widetilde \bI_k^N = \{\bI^N_k, \{ u^{N,n,J^*}_{k-1}\vert_{\partial \Omega_l}, 0\leq n \leq \nfi\} \}.
\end{equation*}
From \cite{guetatPhD, ABGM2017}, there is numerical evidence that the computations of $\DDM(T_N,\Delta T,y^N_k)$ using $\widetilde \bI_k^N$ yield $ [\cE(T_N,\Delta T,y^N_k);\zeta^N_k]$ after a  reduced number of iterations $J^*$ which is independent of $k$. Thus \cost~would be kept constant and
\begin{equation*}
u^{N,\nfi,J^*}_k=\DDM(T_N,\Delta T,y^N_k) \xrightarrow[(\widetilde \bI^N_k,\cost)]{} [\cE(T_N,\Delta T,y^N_k);\zeta^N_k],
\end{equation*}
\ym{where the above ``\cost'' is much small than the cost of the fine solver}.
\subsection{Parareal coupled with iterative high-order time integration schemes}
\label{eq:sdc}
Spectral Deferred Correction (SDC, \cite{DGR2000}) is an iterative time integration scheme. Starting from an initial guess of $u(t)$ at discrete points, the method adds successive corrections to this guess. The corrections are found by solving an associated evolution equation. Under certain conditions, the correction at every step increases by one the accuracy order of the time discretization.

We carry here a simplified discussion on how to build $[\cE(T_N,\Delta T, y^N_k);\zeta^N_k]$ when $\cS=\SDC$ and connect it to the so-called Parallel Full Approximation Scheme in Space-Time (PFASST, \cite{Minion2011,EM2012}). For a given time interval $[T_N,T_{N+1}]$, \ym{let us}  consider its $\nfi+1$ associated Gauss-Lobatto points $\{t_{N,n}\}_{n=0}^{\nfi}$ and quadrature weights $\{\omega_{n}\}_{n=0}^{\nfi}$. The Gauss-Lobatto points are such that $t_{N,0}=T_N$ and $t_{N,\nfi}=T_{N+1}$. \ym{Let us} denote $u^{N,n,j}_k$ the approximation of $u(t_{N,n})$ at parareal iteration $k$ after $j\geq0$ SDC iterations. Assuming that one uses an implicit time-stepping scheme  to solve the corrector equations involved in this method, $u^{N,n+1,j}_k$ is given by
\begin{align*}
u^{N,n+1,j}_k &= u^{N,n,j}_k +(t_{N,n+1}-t_{N,n})\left(\cA(t_{N,n+1},u^{N,n+1,j}_k)-\cA(t_{N,n+1},u^{N,n+1,j-1}_k)\right)\\
\quad & +\sum_{m=0}^{\nfi} \omega_m \cA(t_{N,m},u^{N,m,j-1}_k), \quad 1\leq j,\ 0\leq n \leq\nfi-1, \\
u^{N,0,j}_k &=y^N_k \\
u^{N,n,0}_k &\text{ given for $0\leq n \leq \nfi$.}
\end{align*}
To speed-up computations, one of the key elements is the choice of the starting guesses $u^{N,n,0}_k,\ 0\leq n\leq \nfi$. Without entering into very specific details, the PFASST algorithm is a particular instantiation of the above scheme when $J^*=1$ and $u^{N,n,0}_k$ uses information produced at the previous parareal iteration $k-1$. Therefore, PFASST falls into the present framework in the sense that it produces
\begin{equation*}
u^{N,\nfi,1}_k=\SDC(T_N,\Delta T, y^N_k)
\end{equation*}
with
\begin{equation*}
\widetilde \bI_k^N = \{T_N, \Delta T, y^N_k, \bP_{k-1} \}
\end{equation*}
and it is expected that $u^{N,\nfi,1}_k=[\cE(T_N,\Delta T, y^N_k);\zeta^N_k]$.

An additional component of PFASST is that the algorithm also tries to improve the accuracy of the coarse solver $\cG$ using SDC iterations built with $\widetilde \bI_N^k$. This has not been taken into account in our adaptive parareal algorithm \eqref{eq:paraAP}.

\corr{Another algorithm that progressively improves the quality of the fine solver across iterations is MGRIT, which couples parareal with multigrid iterative methods (see \cite{FFKMS2014, MSBER2015}. The PFASST algorithm has also been coupled with multigrid techniques in \cite{MSBER2015}.}

\subsection{Coupling with internal iterative schemes}
\label{sec:APode}
Any implicit discretization of problem  \eqref{eq:pde} leads to discrete linear or nonlinear systems of equations which are often solved with iterative schemes. When they are involved as internal iterations within the parareal algorithm, one could try to speed them up by building good initial guesses based on information from previous parareal iterations. We illustrate this idea in the simple case where:
\begin{itemize}
\item $\cA(t,\cdot)$ is a linear differential operator in $\bU$ complemented with suitable boundary conditions,
\item we use an implicit Euler scheme for the time discretization.
\end{itemize}
A sequential solution of problem \eqref{eq:pde} with the implicit Euler scheme goes as follows. At each time $t^{N,n}=T_N+n\delta t$, the solution $u(t^{N,n})$ is approximated by the function $u^{N,n}\in \bU$ which is itself the solution to
\begin{equation*}
\cB(u^{N,n}) = g^{N,n},
\end{equation*}
where $g^{N,n} = u^{N,n-1} + \delta t f(t^{N,n})$ and
\begin{equation*}
\cB(v) \coloneqq v+\delta t \cA(t^{N,n},v),\quad \forall v \in \bU.
\end{equation*}
Note that $\cB$ depends on time but our notation does not account for it in order not to overload the notations.

After discretization of $\bU$, the problem classically reduces to solving a linear system of the form
\begin{equation*}
\textbf{B} \bar u^{N,n} = \bar g^{N,n},
\end{equation*}
for the unknown $\bar u^{N,n}$ in some discrete subspace $S$ of $\bU$. Usually, the above system is solved either by means of a conjugate gradient method or by a Richardson iteration of the form
\begin{equation*}
\begin{cases}
&\bar u^{N,n,j} = (\textbf{Id} + \omega \textbf{PB}) \bar u^{N,n,j-1}+\omega \textbf{P} \bar b,\quad j\geq 1 \\
&\bar u^{N,n,0} \in S \text{ given.}
\end{cases}
\end{equation*}
Here, $\omega$ is a suitably chosen relaxation parameter and $\textbf{P}$ can be seen as a pre-conditioner. The internal iterations $j$ are stopped whenever a certain criterion is met (it could be an a posteriori estimator) and we denote by $J_{N,n}$ their final number. Obviously, $J_{N,n}$ depends on the starting guess for which a usual choice is to take the solution at the previous time, that is
\begin{equation*}
\bar u^{N,n,0} = \bar u^{N,n-1,J_{N,n-1}}.
\end{equation*}
In order to achieve our goal, i.e. maintaning a low cost while increasing the accuracy at each parareal step, we  can now reuse information from previous parareal iterations for the starting guess. In \cite{mulaPhD}, two options are explored. The first is
\begin{align*}
&
\begin{cases}
\bar u^{N,n,0}_k = \bar u^{N,n-1,J_{N,n-1,k}}_k, \quad &\text{if $k=0$} \\
\bar u_k^{N,n,0} =  \bar u_{k-1}^{N,n,J_{N,n,k-1}}, \quad &\text{if $k\geq 1$},
\end{cases}\\
\intertext{and the second, less natural choice,}
&
\begin{cases}
\bar u^{N,n,0}_k = \bar u^{N,n-1,J_{N,n-1,k}}_k, \quad &\text{if $k=0$} \\
\bar u_k^{N,n,0} =  \bar u_{k-1}^{N,n,J_{N,n,k-1}}  + \bar u_k^{N,n-1,J_{N,n,k}}- \bar u_{k-1}^{N,n-1,J_{N,n-1,k-1}} , \quad &\text{if $k\geq 1$}.
\end{cases}
\end{align*}
In the first case, we take over the internal iterations at the point where they were stopped in the previous parareal iteration $k-1$. In addition to this, in the second case, the term $u_k^{N,n-1,J_{N,n,k}}- u_{k-1}^{N,n-1,J_{N,n-1,k-1}}$ tries to better take the dynamics of the process into account. Note that the use of solutions that have been produced in the previous parareal iterations is at the expense of additional memory requirements. It might also be at the cost of a certain increase in the complexity locally at certain times. However, \cite{mulaPhD} shows in a restricted setting that these starting guesses (in particular the second) have interesting potential to enhance the speed-up of the parareal algorithm. A general theory on this aspect will be presented in a forthcoming work.

\bibliographystyle{siamplain}
\bibliography{references}

\begin{thebibliography}{10}

\bibitem{ABGM2017}
{\sc S.~Aouadi, D.~Q. Bui, R.~Guetat, and Y.~Maday}, {\em Convergence analysis
  of the coupled parareal-schwarz waveform relaxation method}, 2019.
\newblock In preparation.

\bibitem{baffico2002parallel}
{\sc L.~Baffico, S.~Bernard, Y.~Maday, G.~Turinici, and G.~Z{\'e}rah}, {\em
  Parallel-in-time molecular-dynamics simulations}, Physical Review E, 66
  (2002), p.~057701.

\bibitem{Bal2002}
{\sc G.~Bal and Y.~Maday}, {\em A {``}parareal{"} time discretization for
  non-linear {PDE}'s with application to the pricing of an {A}merican put},
  Recent developments in domain decomposition methods, 23 (2002), pp.~189--202.

\bibitem{bal2008symplectic}
{\sc G.~Bal and Q.~Wu}, {\em Symplectic parareal}, in Domain decomposition
  methods in science and engineering XVII, Springer, 2008, pp.~401--408.

\bibitem{Bal2003}
{\sc {Bal, G.}}, {\em {Parallelization in time of (stochastic) ordinary
  differential equations}}, 2003.
\newblock {Preprint, http://www.columbia.edu/~gb2030/PAPERS/paralleltime.pdf}.

\bibitem{carlberg2019data}
{\sc K.~Carlberg, L.~Brencher, B.~Haasdonk, and A.~Barth}, {\em Data-driven
  time parallelism via forecasting}, SIAM Journal on Scientific Computing, 41
  (2019), pp.~B466--B496.

\bibitem{dai2013symmetric}
{\sc X.~Dai, C.~Le~Bris, F.~Legoll, and Y.~Maday}, {\em Symmetric parareal
  algorithms for hamiltonian systems}, ESAIM: Mathematical Modelling and
  Numerical Analysis, 47 (2013), pp.~717--742.

\bibitem{Dai}
{\sc X.~Dai and Y.~Maday}, {\em Stable parareal in time method for first- and
  second-order hyperbolic systems}, {SIAM J. Sci. Comput.}, 35 (2013),
  pp.~A52--A78.

\bibitem{DP1980}
{\sc J.~R. Dormand and P.~J. Prince}, {\em A family of embedded runge-kutta
  formulae}, Journal of computational and applied mathematics, 6 (1980),
  pp.~19--26.

\bibitem{DGR2000}
{\sc A.~Dutt, L.~Greengard, and V.~Rokhlin}, {\em Spectral deferred correction
  methods for ordinary differential equations}, BIT Numerical Mathematics, 40
  (2000), pp.~241--266.

\bibitem{EM2012}
{\sc M.~Emmett and M.~Minion}, {\em Toward an efficient parallel in time method
  for partial differential equations}, Communications in Applied Mathematics
  and Computational Science, 7 (2012), pp.~105--132.

\bibitem{FFKMS2014}
{\sc R.~D. Falgout, S.~Friedhoff, T.~V. Kolev, S.~P. MacLachlan, and J.~B.
  Schroder}, {\em Parallel time integration with multigrid}, SIAM Journal on
  Scientific Computing, 36 (2014), pp.~C635--C661.

\bibitem{FMOS2017}
{\sc R.~D. Falgout, T.~A. Manteuffel, B.~O'Neill, and J.~B. Schroder}, {\em
  Multigrid reduction in time for nonlinear parabolic problems: A case study},
  SIAM Journal on Scientific Computing, 39 (2017), pp.~S298--S322.

\bibitem{Farhat2003}
{\sc C.~Farhat and M.~Chandesris}, {\em Time-decomposed parallel
  time-integrators: theory and feasibility studies for fluid, structure, and
  fluid--structure applications}, International Journal for Numerical Methods
  in Engineering, 58 (2003), pp.~1397--1434.

\bibitem{GG2018}
{\sc M.~Gaja and O.~Gorynina}, {\em Parallel in time algorithms for nonlinear
  iterative methods}, 2018.
\newblock To appear in ESAIM Proceedings of CEMRACS 2016 -- Numerical
  Challenges in Parallel Scientific Computing.

\bibitem{Gander2015}
{\sc M.~J. Gander}, {\em 50 years of time parallel time integration}, in
  Householder Symposium XIX June 8-13, Spa Belgium, 2015, p.~81.

\bibitem{GH2008}
{\sc M.~J. Gander and E.~Hairer}, {\em Nonlinear convergence analysis for the
  parareal algorithm}, in Domain Decomposition Methods in Science and
  Engineering XVII, Springer, 2008, pp.~45--56.

\bibitem{guetatPhD}
{\sc R.~Guetat}, {\em {M\'ethode de parall\'elisation en temps: Application aux
  m\'ethodes de d\'ecomposition de domaine}}, PhD thesis, {Paris VI}, 2012.

\bibitem{HW1996}
{\sc E.~Hairer and G.~Wanner}, {\em Solving Ordinary Differential Equations II.
  Stiff and Differential-Algebraic Problems}, vol.~14, 01 1996.

\bibitem{HW1999}
{\sc E.~Hairer and G.~Wanner}, {\em Stiff differential equations solved by
  radau methods}, Journal of Computational and Applied Mathematics, 111 (1999),
  pp.~93 -- 111.

\bibitem{covid-model-2020}
{\sc Q.~Lin, S.~Zhao, D.~Gao, Y.~Lou, S.~Yang, S.~S. Musa, M.~H. Wang, Y.~Cai,
  W.~Wang, L.~Yang, and D.~He}, {\em A conceptual model for the coronavirus
  disease 2019 (covid-19) outbreak in wuhan, china with individual reaction and
  governmental action}, International Journal of Infectious Diseases, 93
  (2020).

\bibitem{Maday2001}
{\sc J.~Lions, Y.~Maday, and G.~Turinici}, {\em {R{\'e}solution d'{EDP} par un
  sch{\'e}ma en temps parar{\'e}el}}, C. R. Acad. Sci. Paris,  (2001).
\newblock t. 332, S{\'e}rie I, p. 661-668.

\bibitem{MRS}
{\sc Y.~Maday, M.~Ronsquist, E., and G.~Staff}, {\em The parareal in time
  algorithm: Basics, stability analysis and more}, in Staff PhD Thesis, see
  below (2006), pp.~653--663.

\bibitem{MST2007}
{\sc Y.~Maday, J.~Salomon, and G.~Turinici}, {\em Monotonic parareal control
  for quantum systems}, SIAM Journal on Numerical Analysis, 45 (2007),
  pp.~2468--2482.

\bibitem{madayTuriniciDDM}
{\sc Y.~Maday and G.~Turinici}, {\em {The Parareal in Time Iterative Solver: a
  Further Direction to Parallel Implementation}}, in {Domain Decomposition
  Methods in Science and Engineering}, Springer Berlin Heidelberg, 2005,
  pp.~441--448.

\bibitem{Minion2010}
{\sc M.~Minion}, {\em A hybrid parareal spectral deferred corrections method},
  Comm. App. Math. and Comp. Sci., 5 (2010).

\bibitem{Minion2011}
{\sc M.~Minion}, {\em A hybrid parareal spectral deferred corrections method},
  Communications in Applied Mathematics and Computational Science, 5 (2011),
  pp.~265--301.

\bibitem{MSBER2015}
{\sc M.~L. Minion, R.~Speck, M.~Bolten, M.~Emmett, and D.~Ruprecht}, {\em
  Interweaving {PFASST} and parallel multigrid}, SIAM Journal on Scientific
  Computing, 37 (2015), pp.~S244--S263.

\bibitem{MW2008}
{\sc M.~L. Minion, A.~Williams, T.~E. Simos, G.~Psihoyios, and C.~Tsitouras},
  {\em Parareal and spectral deferred corrections}, in AIP Conference
  Proceedings, vol.~1048, 2008, p.~388.

\bibitem{mulaPhD}
{\sc O.~Mula}, {\em {Some contributions towards the parallel simulation of time
  dependent neutron transport and the integration of observed data in real
  time}}, PhD thesis, {Paris VI}, 2014.

\bibitem{Nievergelt1964}
{\sc J.~Nievergelt}, {\em Parallel methods for integrating ordinary
  differential equations}, Commun. ACM, 7 (1964), pp.~731--733.

\bibitem{Petzold1983}
{\sc L.~Petzold}, {\em Automatic selection of methods for solving stiff and
  nonstiff systems of ordinary differential equations}, SIAM journal on
  scientific and statistical computing, 4 (1983), pp.~136--148.

\bibitem{quarteroni1996domain}
{\sc A.~Quarteroni and A.~Valli}, {\em Domain decomposition methods for partial
  differential equations}, Von Karman institute for fluid dynamics, 1996.

\bibitem{TW2005}
{\sc A.~Toselli and O.~Widlund}, {\em Domain decomposition methods: algorithms
  and theory}, vol.~3, Springer, 2005.

\end{thebibliography}
\end{document}